\documentclass[final,leqno,onefignum,onetabnum]{siamltex1213}
\usepackage{graphicx}
\usepackage{amssymb}
\usepackage{amsmath}
\usepackage{subfigure}
\usepackage{float}

\title{Existence and stability in a virtual interpolation method of the Stokes equations}

\author{
Seong-Kwan Park \footnotemark[1]\ \footnotemark[2]
\and Gahyung Jo\footnotemark[1]\ \footnotemark[3]
\and Hi Jun Choe \footnotemark[1]
}

\begin{document}
\maketitle
\renewcommand{\thefootnote}{\fnsymbol{footnote}}

\footnotetext[1]{Department of Mathematics, Yonsei University, Seoul, 120-749, Republic of Korea.}
\footnotetext[2]{Department of Turbulent Boundary Layer, PARK Seong-Kwan Institute, Seoul 136-858, Republic of Korea.}
\footnotetext[3]{National Fusion Research Institute, Daejeon 169-148, Republic of Korea.}

\renewcommand{\thefootnote}{\arabic{footnote}}

\newtheorem{thm}{Theorem}[section]
\newtheorem{prop}[thm]{Proposition}
\newtheorem{lem}[thm]{Lemma}
\newtheorem{cor}[thm]{Corollary}
\newtheorem{example}[thm]{Example}
\newtheorem{conj}[thm]{Conjecture}
\newtheorem{note}[thm]{Note}
\newtheorem{remark}[thm]{Remark}

\slugger{mms}{xxxx}{xx}{x}{x--x}

\begin{abstract}
In this paper, we propose a new virtual interpolation point method
to formulate the discrete Stokes equations.
We form virtual staggered structure for the velocity and pressure
from the actual computation node set. The virtual interpolation
point method by a point collocation scheme is well suited to
meshfree scheme since the approximation comes from smooth kernel and
we can differentiate directly the kernels. The focus of this paper
is laid on the contribution to a stable flow computation without
explicit structure of staggered grid. In our method, we don't have
to construct explicitly the staggered grid at all. Instead, there
exists only virtual interpolation points at each computational node
which play a key role in discretizing the conservative quantities of
the Stokes equations.

We prove the inf-sup condition for virtual interpolation point
method with virtual structure of staggered grid and the existence
and stability of discrete solutions.
\end{abstract}



\pagestyle{myheadings}
\thispagestyle{plain}

\section{Introduction}
Despite the fact that there have been lots of schemes to solve flow
problems, for example, the incompressible Navier-Stokes flow, the
Euler flow which is compressible or incompressible, and the
compressible Navier-Stokes,  the issues on the stability, the
efficiency and the accuracy take place frequently as the complexity
of the problem increases. The finite difference method which has
long history uses the staggered grids for the velocity and pressure
for the purpose of avoiding the stability issue.

We are concerned with existence and stability issues for the
numerical approximation of the stationary incompressible Stokes
equation by virtual interpolation point(VIP) method derived from
meshfree scheme. For the finite element, there are extensive works
for inf-sup stability like Babuska\cite{babuska1973},
Brezzi\cite{brezzi1974} and Girault and Raviart\cite{girault1986}.
We form virtual interpolation point grid for the velocity and
pressure to exploit the inf-sup stability of staggered structure and
then from the interpolation using collocation we prove the existence
of discrete solution. we think our idea combining the virtual
staggered structure and interpolation is very powerful to solve many
difficult fluid problems.

The meshfree scheme has been successfully applied to various
problems in fluid as shown in Choe \textit{et al.}~\cite{choe2006b},
Park \textit{et al.}~\cite{park2007}, and Park~\cite{park2010}. One
of the significant features of meshfree scheme is the versatile
property of reproducing kernel like complete local generation of
polynomials. In this paper we adopt point collocation method to
formulate the discrete Stokes equations. The point collocation
method is well suited to meshfree scheme since the approximation
comes from smooth kernel and we can differentiate directly the
kernels. For more details of basis function (shape function),
$\Psi$, we refer Liu \textit{et al.}~\cite{liu1997}. We include
several numerical results to confirm our theory.

For simplicity, we consider two dimensional stationary Stokes problem with periodic boundary condition,
    \begin{align}
        \label{stokes}
        -\Delta \mathbf{u} + \nabla p &=  \mathbf{f},    \\
        \nabla\cdot \mathbf{u} &= 0,\nonumber
    \end{align}
in the unit square domain $\Omega$, where $\mathbf{u}$ is velocity,
$p$ is pressure and $\mathbf{f}$ is external force. From
Helmholtz-Weyl decomposition, when $\mathbf{f}\in L^2(\Omega)$, we
have that $\mathbf{f} = \nabla a + \mathbf{d}$, $\mbox{div}
\mathbf{d} =0$ weakly in $L^2$. Therefore, by merging $\nabla a$ to
pressure, we can assume $\mathbf{f}$ is solenoidal in
(\ref{stokes}). Furthermore taking divergence we may assume the
pressure $p$ is harmonic in (\ref{stokes}) although we do not need
harmonicity in formulation, namely,

\begin{equation*}
	\Delta p = 0.
\end{equation*}

Let $X = H^1_{per}(\Omega) = \lbrace \mathbf{u} : \text{$\mathbf{u}$
is periodic,}\int_\Omega \mathbf{u} d\mathbf{x} = 0 \text{ and }
\|\mathbf{u}\|_X^2  =\int_\Omega |\nabla \mathbf{u} |^2 d\mathbf{x} <
\infty \rbrace$ and $M = L^2_{per}(\Omega) = \lbrace q : \text{$q$
is periodic and } \int_\Omega |q|^2  d\mathbf{x} < \infty\rbrace$.
By the saddle point argument for the function space $X \times M$, the
existence of the solution to the Stokes equations follows from the
inf-sup condition as long as
 $\mathbf{f}\in H^{-1}_{per}(\Omega)$.
\begin{definition}
    $X\times M$ satisfies inf-sup condition for a bilinear form $b$
     if there is a positive constant $\mu > 0$ such that
\begin{equation*}
    \inf_{p\in M\backslash \{0\}} \sup_{\mathbf{u}\in X} \frac{b(\mathbf{u}, p)}{\|\mathbf{u}\|_X \|p\|_M} \ge \mu > 0.
\end{equation*}
\end{definition}

\begin{thm}
Suppose that  $X\times M$ satisfies inf-sup condition for a bilinear
form $b$. Given $\mathbf{f} \in X^\prime$, there is a pair
$(\mathbf{u},p) \in X \times M$ such that
\begin{align*}
    a(\mathbf{u}, \mathbf{v}) + b(\mathbf{v} ,p) &= \langle \mathbf{f}, \mathbf{v} \rangle~~~\forall \mathbf{v} \in X, \\
    b(\mathbf{u}, q) &= 0 ~~~\forall q \in M,
\end{align*}
where $a(\mathbf{u}, \mathbf{v})=\int_\Omega \nabla\mathbf{u}\cdot
\nabla\mathbf{v} d\mathbf{x}$ and $b(\mathbf{u}, q)=\int_\Omega
\mathrm{div}\mathbf{u}q d\mathbf{x}$. Moreover $(\mathbf{u},p)$ satisfies
$$
\|\mathbf{u}\|_{X} +\|p\|_M \leq C \|\mathbf{f}\|_{X^{\prime}}.
$$
for a constant $C>0$.

\end{thm}
We discretize the incompressible Stokes equations by meshfree
scheme. Then by the inf-sup condition for discrete version in
Theorem \ref{theorem:inf-sup_for_vip}, we prove existence and
stability of VIP method. The most important contribution in this
paper is the single node scheme for both velocity and pressure by
VIP method. As a natural consequence, the computation becomes very
efficient and stable and is very robust to geometrical complexity.
Although the approximation node set may not have any structural
condition, the numerical stability follows from the facts that VIP
method compromise the usual staggered grid and that any discrete
vector can be reproduced by meshfree scheme. Since the collocation
method requires the pointwise evaluation of the second derivatives
at each node, we need higher regularity on the external force
$\mathbf{f} \in C^\alpha$ to get approximation error. Theorem
\ref{theorem:inf-sup_for_vip} and \ref{theorem:stability} are our
main theorems for existence and stability.

To validate VIP method, we conduct several numerical simulations.


\section{Formulation of VIP method}
First we introduce the meshfree method in view of moving least square by general setting and then consider the periodic domain.
We let $\Omega\subset \mathbb{R}^n$ and $u$ be a bounded $C^\infty$ function.
We consider the set of polynomials of degree less than $m$
    \begin{equation}    \label{equation:polynomial}
        P_m = \{ x_1^{\alpha_1}\cdots x_n^{\alpha_n}: |\alpha
        |=\alpha_1+\cdots+\alpha_n \leq m \},
    \end{equation}
and introduce window function $\Phi$ a nonnegative smooth function with compact support.
By minimizing the local error residual function
	\begin{equation*}
		J(\mathbf{a}(\bar{\mathbf{x}}))= \int_\Omega \left|u(\mathbf{x}) -
		P_m\left(\frac{\mathbf{x}-\bar{\mathbf{x}}}{\rho}\right)\cdot {\bf
		a}(\bar{\mathbf{x}})\right|^2
		\frac{1}{\rho^n}\Phi\left(\frac{\mathbf{x}-\bar{\mathbf{x}}}{\rho}\right)
		d\mathbf{x},
	\end{equation*}
for a positive dilation parameter $\rho$ and setting
$\bar{\mathbf{x}} =\mathbf{x}$, we obtain the continuous projection
$Ku$ of $u$
	\begin{equation}
		\label{globalapproximation}
		 Ku(\mathbf{x}) = \int k_\rho(\mathbf{x}-\mathbf{y},\mathbf{x}) u(\mathbf{y}) d\mathbf{y},
		\end{equation}
by a reproducing kernel $k_\rho$ (see equation (3) in \cite{choe2006a}).
We note that in periodic domain the kernel function $k_\rho(\mathbf{z},\mathbf{x})$ is independent of $\mathbf{x}$ and $Ku$ is the usual convolution of $k_\rho$ and $u$.
The key merit of meshfree scheme is the reproducing property of polynomials of degree $m$. For a more detail, we refer \cite{liu1996}.
Furthermore there is a mathematical theorem interpreting the interpolation errors and numerical convergence.
	\begin{thm}[see~\cite{choe2006a}]
		\label{theorem:diff_interpolation}
		Suppose the boundary of $\Omega$ is smooth and $\mathrm{supp}
		k_\rho \cap \overline{\Omega}$ is convex.
		If $m$ and $p$ satisfy
		\begin{equation*}
			m >\frac{n}{p} - 1,
		\end{equation*}
then the following interpolation estimate of the projection holds
		\begin{equation*}
        	\| D^\beta v - D^{\beta} Kv \|_{L^p(\Omega)} \leq C(m)\rho^{m+1-|\beta|}
        	\|v\|_{W^{m+1,p}(\Omega)},
		\end{equation*}
for all $0 \leq |\beta| \leq m$.
\end{thm}
Now let us consider the discrete problem. Let $R = \lbrace \mathbf{x}_I : I = 1,2,\cdots, N \rbrace$ be a regular node set.
For given computation node $\mathbf{x}_I \in R$, we obtain the shape function $\Psi_I$ from moving least square reproducing kernel (MLSRK) method by Liu and Belytschko~\cite{liu1997}. In fact the approximation is a linear combination of shape functions $\Psi_I$ for given node point $I$ and define the discrete projection operator $\Gamma$ by
    \begin{equation*}
        \Gamma u = \sum_{I} \Psi_I u(\mathbf{x}_I).
    \end{equation*}
The exact form of discrete shape function due to Liu and
Belytschko\cite{liu1997} is
    \begin{align*}
        M(\mathbf{x})&=\sum_I P_m^t \left(\frac{\mathbf{x}-\mathbf{x}_I}{\rho}\right)
        P_m\left(\frac{\mathbf{x}-\mathbf{x}_I}{\rho}\right)
        \frac{1}{\rho^n}\Phi\left(\frac{\mathbf{x}-\mathbf{x}_I}{\rho}\right),\\
        \Psi_I(\mathbf{x})& =
        P_m(0)[M(\mathbf{x})]^{-1}P_m^t\left(\frac{\mathbf{x}-\mathbf{x}_I}{\rho}\right)
        \frac{1}{\rho^n} \Phi\left(\frac{\mathbf{x}-\mathbf{x}_I}{\rho}\right).
    \end{align*}
Note that a polynomial of degree less than $m$ is exactly reproduced by the discrete projection $\Gamma$.
If we consider discrete problems, the point collocation method is well suited to meshfree
scheme since the basis functions are differentiable at all orders and they can reproduce any polynomials locally at given degree.
We need only to differentiate the basis functions according to the partial differential equations.

Now we study the Stokes equations.
Let $\Psi_I$ be shape function at node $\mathbf{x}_I\in R$.
Define virtual collocation point set $T= \{\mbox{y}_J: J=1,...,M\}$.
We are looking for an approximate solution
    \begin{equation*}
        (\mathbf{u},p)=\left( \sum_{I=1}^N \Psi_I \mathbf{u}_I, \sum_{I=1}^N \Psi_I p_I
        \right),
    \end{equation*}
to the discrete Stokes equations in the context of point collocation at each virtual node point $\mathbf{y}_J\in T$,
\begin{align*}
    -\Delta \mathbf{u}(\mathbf{y}_J) + \nabla p(\mathbf{y}_J)
    &= \mathbf{f}(\mathbf{y}_J) , \\
    \nabla\cdot \mathbf{u }(\mathbf{y}_J) &= 0.
\end{align*}
An important fact in our point collocation scheme is that the virtual interpolation point set $T$ is not necessarily the node point sets $R$.
Indeed, we are going to evaluate velocity and pressure coefficients from discrete Stokes equations at virtual interpolation points in $T$ which are collocation points.
Therefore we have a great freedom to choose node sets.

For simplicity we assume $2D$ case.
We denote numerical derivatives by using multi-index $\alpha$,
    \begin{equation*}
        \mathcal{D}^{\alpha} u(\mathbf{x}_J) = \sum_{I=1}^N\Psi_{I}^{[\alpha]}(\mathbf{x}_J)
        u_I, \quad \mathbf{x}_J\in T,
    \end{equation*}
where $\mathcal{D}^\alpha$ means $\alpha$-th numerical derivatives, $\mathcal{D}^{[2,0]} = \frac{\partial^2}{\partial x_1^2}$,
$\mathcal{D}^{[0,2]} = \frac{\partial^2}{\partial x_2^2}$, $\mathcal{D}^{[1,0]} = \frac{\partial}{\partial x_1}$,
$\mathcal{D}^{[0,1]} = \frac{\partial}{\partial x_2}$ and $\mathcal{D}^{[0,0]}$ means identity.
We write the discrete incompressible Stokes equations in matrix form,
	\begin{align*}
    	-AU +GP &= F,    \\
	    DU &= 0,
	\end{align*}
when we denote $u_I = u(\mathbf{x}_I)$, $v_I = v\left(\mathbf{x}_I\right)$, $p_I = p(\mathbf{x}_I)$, and $f_{1,I} = f_1\left(\mathbf{y}_I\right)$, $f_{2,I} = f_2\left(\mathbf{y}_I\right)$ for $\mathbf{x}_J\in R$ and $\mathbf{y}_I\in T$, and we have
\begin{align*}
    U &=
    \begin{pmatrix}
        u_1 & u_2 & \cdots & u_N & v_1 & v_2 & \cdots & v_N
    \end{pmatrix}^t,
    \\
    P &=
    \begin{pmatrix}
        p_1 & p_2 & \cdots & p_N
    \end{pmatrix}^t,
    \\
    F &=
    -\begin{pmatrix}
        f_{1,1} & f_{1,1} & \cdots & f_{1,M} & f_{2,1} & f_{2,2} & \cdots & f_{2, M}
    \end{pmatrix}^t.
\end{align*}
The stiffness matrix $A, G$ and $D$ matrix are following:
\begin{align*}
    A =
    \begin{pmatrix}
       \Delta_h \Psi^{[0,0]}_{M\times N} &\mathbf{0}_{M\times N} \\
        \mathbf{0}_{M\times N} & \Delta_h \Psi^{[0,0]}_{M\times N}
    \end{pmatrix}_{2M\times 2N},  \qquad
    G =
    \begin{pmatrix}
         \Psi^{[1,0]}_{M\times N}   \\
         \Psi^{[0,1]}_{M\times N}
    \end{pmatrix}_{2M\times N},
\end{align*}
where $\Delta_h$ the Laplace operator in finite difference type, the $I,J$ component of the matrix,
\begin{align*}
    \left(\Psi^{[\alpha,\beta]}_{M\times N}\right)_{IJ}  = \Psi^{[\alpha,\beta]}_J\left(\mathbf{x}_I\right)
\end{align*}

We introduce the virtual interpolation point for matrix $D^*$ corresponding to the discrete divergence operator.
The virtual interpolation points for virtual staggered grid points for velocity field and pressure at virtual interpolation point of node $\mathbf{x}_I \in T$ are:

\begin{align*}
\mathbf{z}_{I,1}^{+} &= \mathbf{x}_I+(h/2,0),\qquad
\mathbf{z}_{I,1}^{-} = \mathbf{x}_I-(h/2,0),\\
\mathbf{z}_{I,2}^{+} &= \mathbf{x}_I+(0,h/2),\qquad
\mathbf{z}_{I,2}^{-} = \mathbf{x}_I-(0,h/2),
\end{align*}
and define the discrete divergence
\begin{align*}
    (D^* U)_I &=\\
       & \frac{1}{h} \sum_{J=1}^N \left[\Psi_{J}^{[0,0]}(\mathbf{z}_{I,1}^{+}) - \Psi_{J}^{[0,0]}(\mathbf{z}_{I,1}^{-})\right]u_J+
        \frac{1}{h} \sum_{J=1}^N \left[\Psi_{J}^{[0,0]}(\mathbf{z}_{I,2}^{+}) -
        \Psi_{J}^{[0,0]}(\mathbf{z}_{I,2}^{-})\right]v_{J}.
\end{align*}
So we can write the numerical dual operator of divergence by numerical derivative matrix $D$,
\begin{align*}
    D=
    \begin{pmatrix}
    D_h \Psi^{[0,0]}_{M\times N} \\
    D_h \Psi^{[0,0]}_{M\times N}
    \end{pmatrix}_{2M\times N},
\end{align*}
where $D_h$ means finite difference operator.

\begin{remark}
    By adopting periodic boundary condition, we can extend to whole plane.
\end{remark}

We employ the discrete divergence operator $D^*$ to define the
discrete Laplace operator $AU= D^*(D U)$ instead of stiffness matrix
$A$. Instead of the gradient matrix $G P$ of the pressure, we
formulate the velocity equations by $D P$. But the replacement is
simply for the convenience of analysis and the existence proof will
hold for $GP$ after considering projection error, too.

\section{Existence and stability}
Now we prove that the virtual point collocation scheme is stable for the Stokes flow
    \begin{align}       \label{discretestokes}
        -AU +D P &=F,  \\
        D^* U &= 0,\nonumber
    \end{align}
when the approximation node set $\{ \mathbf{x}_I\}=R$ is
sufficiently dense. To be more specific we introduce a definition.
	\begin{definition}{(\bf Realization)}
		The node set $R=\lbrace\mathbf{x}_I, I=1,...,N \rbrace$ realizes the set of virtual interpolation point $T=\lbrace \mathbf{y}_J, J=1,...,M\rbrace$ if for each $U\in \mathbb{R}^{M}$ there is $u\in \mathbb{R}^N$ such that
    	\begin{equation*}
        	U_{J} = \sum_{I=1}^{N} \Psi_I(\mathbf{y}_J)u_I.
   	 \end{equation*}
	\end{definition}

We find that the number of element $N$ of approximation node set $R$ must be greater than or equal to the number of element $M$ of  the virtual collocation point set $T$ for the realization.
Moreover the representation is not unique if there are sufficiently more approximation nodes than the virtual interpolation point nodes.
Therefore we can not have uniqueness of solution but the existence is guaranteed by the following inf-sup stability theorem.
We assume our virtual interpolation point set $T$ is regular grid so that the nodes are lattice points $\{(kh,jh)\}$, where $k$ and $j$ are integers and the edge length $h$ is a positive number.
	\begin{thm}\label{theorem:inf-sup_for_vip}
		We let the virtual collocation point set
		$S=\lbrace\mathbf{z}^{\pm}_{J,i}: i=1,2, J=1,\cdots,M \rbrace =\lbrace \mathbf{z}_J, J=1,\cdots,2M \rbrace$ and virtual node point set $T=\lbrace \mathbf{y}_I : I = 1, \cdots, M \rbrace$ form virtual staggered structure(See Fig. \ref{fig:staggered grids}).
		Suppose that $R=\lbrace\mathbf{x}_I\rbrace$ realizes the regular virtual collocation point sets $S$ and $T$.
		Then there is a positive $\mu>0$ independent of $h$ satisfying the inf-sup condition due to Ladyzhenskaya-Brezzi-Babuska such that
    \begin{equation}   \label{discreteinfsup}
        \inf_P \sup_U \langle D^* U, P\rangle  \geq \mu \| P\|_{l^2} \| D U
        \|_{l^2},
    \end{equation}
	\end{thm}

\begin{proof}
Suppose $P$ is an arbitrary vector in $\mathbb{R}^M$ corresponding to the regular node point set $T=\lbrace \mathbf{y}_I, I=1,...,M \rbrace$.
To use integral, we recall the extension pressure $\overline{P}$ that is  piecewise constant corresponding to discrete pressure $P$ such that
    \begin{equation*}
        \overline{P}(\mathbf{z}) = P_I, \quad\mbox{if}\quad |z_1-y_{I,1}| < \frac{h}{2} \quad \mbox{and}\quad |z_2- y_{I,2}| < \frac{h}{2}.
    \end{equation*}
Since $\overline{P}\in L^2(\Omega)$ and the domain is square, there is $\mathbf{v}=(v_1,v_2) \in H^1_{per}(\Omega)$ satisfying
    \begin{equation*}
        \mbox{div} \mathbf{v} = \overline{P}\quad\mbox{and}\quad \|\nabla
        \mathbf{v}\|_{L^2} \leq C \| \overline{P}\|_{L^2},
    \end{equation*}
for a constant $C$. Since we consider periodic domain, we may assume $\int_\Omega \mathbf{v} d\mathbf{x} = 0$.
Since the virtual collocation point set of velocity and pressure form a virtual staggered structure, we have a  discrete velocity $\lbrace V_{IJ}^\pm \rbrace: I=1,\cdots, M, J=1,2. \rbrace$ such that
	\begin{align*}
		 V_{I,1}^{+}&=\int_{0}^{1}v_1(\mathbf{y}_I +(0,ht)) dt,\\
		V_{I,1}^{-}&=\int_{0}^{1}v_1(\mathbf{y}_I +(0,-ht)) dt,\\
		V_{I,2}^{+}&=\int_{0}^{1}v_2(\mathbf{y}_I +(ht,0)) dt,\\
		V_{I,2}^{-}&=\int_{0}^{1}v_2(\mathbf{y}_I +(-ht,0)) dt,
	\end{align*}
where $h$ is the edge length of grid partition and $\mathbf{y}_I\in T$.
If we define the discrete area element
$A_I= \lbrace(z_1, z_2): |z_1-y_{I,1}|<h/2,
|z_2-y_{I,2}|<h/2\rbrace$, then
    \begin{align*}
        \frac{V_{I,1}^{+} -V_{I,1}^{-}}{h} = \frac{1}{h^2} \int_{A_I}
        \frac{\partial v_1}{\partial y_1} dA,\quad \frac{V_{I,2}^{+}
        -V_{I,1}^{-}}{h} = \frac{1}{h^2} \int_{A_I} \frac{\partial v_2}{\partial y_2}
        dA,
\end{align*}
and thus we have
    \begin{align*}
        \int_\Omega \mbox{div}\mathbf{v} \overline{P} dz_1dz_2 &=h
        \sum_{I=1}^{M} \left[\left(V_{I,1}^+ -V_{I,1}^-\right)+\left(V_{I,2}^+ -V_{I,2}^-\right)\right]P_I \\
        &= h^2 \langle D^* V,P \rangle= h^2 \sum_{I=1}^{M} |P_I|^2=h^2\|P\|^2_{l^2} .
    \end{align*}
From H\"older inequality we have
    \begin{equation*}
        \left| \frac{V_{I,i}^{+} -V_{I,i}^{-}}{h}\right|^2 =
        \left|\frac{1}{h^2} \int_{A_I} \frac{\partial v_i}{\partial z_i}
        dA\right|^2 \leq \frac{1}{h^2} \int_{A_I} \left|\frac{\partial v_i}{\partial z_i}\right|^2
        dA,
    \end{equation*}
for $i=1, 2$ and
$$ h^2 \|D V\|^2_{l^2} \leq C \int_\Omega |\nabla
v|^2 \leq C h^2 \|P\|_{l^2}^2.
$$

Considering all terms, we prove the discrete inf-sup condition (\ref{discreteinfsup}).
For the proof of inf-sup condition of staggered grid for finite difference scheme, we refer \cite{shin2009}.
Therefore, the existence of discrete solution vector $(U,P)$ to (\ref{discretestokes}) follows from inf-sup condition.

It remains to show that any virtual velocity vector $\lbrace V_I; I=1,...,2M \rbrace$ can be realized by the real node velocity vector $\lbrace \mathbf{u}_I \rbrace$ on $R$ by interpolation.
Since we are assuming that $R=\lbrace\mathbf{x}_I, I=1,...,N\rbrace$ realizes the regular velocity virtual node sets $S$ and $T$, any vector $(V,P)$ can be written as
    \begin{equation*}
        V_I =\left( \sum_J \Psi_J(\mathbf{z}_I) u_J, \sum_J
        \Psi_J(\mathbf{z}_I) v_J \right) \quad\mbox{and}\quad P_I=\sum_J
        \Psi_J(\mathbf{y}_I) p_J
    \end{equation*}
for all $\mathbf{z}_I\in S$ and $ \mathbf{y}_I\in T$.
\end{proof}

As a corollary, we have the existence of approximate solution.
\begin{cor} \label{corrollary:approximation}
We suppose that all the node sets satisfy the conditions in Theorem \ref{theorem:inf-sup_for_vip}. Then, there exists an approximate solution
    \begin{equation*}
        (U,P)=\left( \sum_{I=1}^N \Psi_I \mathbf{u}_I, \sum_{I=1}^N \Psi_I p_I
        \right),
    \end{equation*}
such that $\lbrace (\mathbf{u}(\mathbf{z}_I),p(\mathbf{y}_J));I=1,\cdots,2M,
J=1,\cdots,M\rbrace=(U, P)$ is solution to (\ref{discretestokes}).
Moreover $(U,P)$ satisfies
    \begin{equation*}
        \|D U\|_{l^2}+\|P\|_{l^2} \leq C \|F\|_{l^2},
    \end{equation*}
for a constant $C$.
\end{cor}

\begin{figure}[H]
        \centering
        \includegraphics[width=0.4\textwidth]{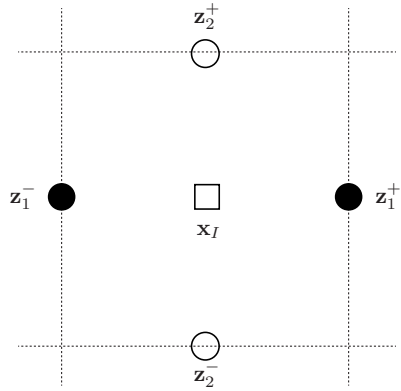}
        \caption{The virtual collocation points $\mathbf{z}_{J,i}^{\pm}$ and the virtual node point $\mathbf{x}_I$.}
        \label{fig:staggered grids}
\end{figure}

For the stability and convergence of virtual interpolation scheme, we assume that
dilation parameter $\rho$ of window function of Theorem \ref{theorem:diff_interpolation} is
comparable to the node interval $h$, namely, there is $C$ satisfying
    \begin{equation*}
        0<\rho <Ch.
    \end{equation*}
If we let $(\mathbf{v},q)$ the true solution in
$H^1_{per}(\Omega)\times L^2_{per}(\Omega)$, then
    \begin{equation*}
        \|\mathbf{v}\|_{H^1(\Omega)}+ \|q\|_{L^2(\Omega)} \leq C\|\mathbf{f}
        \|_{H^{-1}(\Omega)}.
    \end{equation*}
In case of periodic domain with regular node set, the continuous
projection operator $K$ is a convolution of the kernel $k_\rho$ (see equation \eqref{globalapproximation}) and
thus we have
	\begin{equation*}
		-\Delta K \mathbf{v} +\nabla K q =  K \mathbf{f}
		\quad\mbox{and}\quad\mbox{div} K\mathbf{v}=0,
	\end{equation*}
and the stability of continuous projection follows from the energy estimate:

\begin{thm}
	Let $(\mathbf{u}, q) \in H_{per}^1(\Omega) \times L^2_{per}(\Omega)$, $(\mathbf{v},q)$ are a solution of \eqref{stokes} and $K$ is the continuous projection operator in (\eqref{globalapproximation}) then we have an inequality:
    \begin{equation*}
        \| K \mathbf{v}-\mathbf{v}\|_{H^1}+ \| K q - q \|_{L^2} \leq
       C \| K\mathbf{f}-\mathbf{f}\|_{H^{-1}}.
    \end{equation*}
\end{thm}

The analysis of discrete projection $\Gamma$ for $\mathbf{v}$ and
$p$ is more complicated. Let us assume that the reproducing degree $m$ in \eqref{equation:polynomial} is greater than or equal to 2 and polynomials of degree two can be reproduced.

We suppose $\mathbf{v} \in C^{2,\alpha}$ for a $\alpha >0$. For
fixed $\mathbf{x}_I$, we have Taylor expansion if
$|\mathbf{x}-\mathbf{x}_I| \leq \rho\leq Ch$:
$$ \mathbf{v}(\mathbf{x})=\mathbf{v}(\mathbf{x}_I)+\nabla \mathbf{v}(\mathbf{x}_I)(\mathbf{x}-\mathbf{x}_I)
+\frac{1}{2}\nabla^2
\mathbf{v}(\mathbf{x}_I)(\mathbf{x}-\mathbf{x}_I)^2+C\|\nabla^2
\mathbf{v}\|_{C^\alpha}O(h^{2+\alpha}), $$ and from the reproducing
property
    \begin{equation*}
        \Gamma \mathbf{v}(\mathbf{x})= \mathbf{v}(\mathbf{x}_I)+\nabla
        \mathbf{v}(\mathbf{x}_I)(\mathbf{x}-\mathbf{x}_I)
        +\frac{1}{2}\nabla^2
        \mathbf{v}(\mathbf{x}_I)(\mathbf{x}-\mathbf{x}_I)^2+C\| \nabla^2
        \mathbf{v}\|_{C^\alpha}O(h^{2+\alpha}),
    \end{equation*}
and
    \begin{equation*}
        (A \Gamma \mathbf{v})_I =(D^* D (\Gamma\mathbf{v}) )_I = \Delta
        \mathbf{v}(\mathbf{x}_I) + O(h^\alpha).
    \end{equation*}
We have that
\begin{align*}
 \Gamma v_{1,x_1}(\mathbf{x}) &=
    v_{1,x_1} (\mathbf{x}_I)
        +v_{1,x_1 x_1}(\mathbf{x}_I)(x_1- x_{I,1})
        + v_{1,x_1 x_2}(\mathbf{x}_I)(x_2-x_{I,2})+   C\| \nabla^2
        \mathbf{v}\|_{C^\alpha}O(h^{1+\alpha})\\
   \Gamma v_{2,x_2}(\mathbf{x}) &=
   v_{2,x_2} (\mathbf{x}_I)
        +v_{2,x_2 x_2}(\mathbf{x}_I)(x_2- x_{I,2})
        + v_{2,x_1 x_2}(\mathbf{x}_I)(x_1-x_{I,1})+   C\| \nabla^2
        \mathbf{v}\|_{C^\alpha}O(h^{1+\alpha}),
\end{align*}
and we also have that, from divergence free condition,
\begin{align*}
 v_{1,x_1 x_2}(\mathbf{x}_I)(x_2- x_{I,2}) &+v_{2,x_2 x_2}(\mathbf{x}_I)(x_2- x_{I,2})=0\\
v_{2,x_1 x_2}(\mathbf{x}_I)(x_1- x_{I,1}) &+v_{1,x_1 x_1}(\mathbf{x}_I)(x_1- x_{I,1})=0.
\end{align*}
 Taking divergence of $\Gamma \mathbf{v}$ and
noting that $\nabla (\mbox{div} \mathbf{v})(\mathbf{x}_I)=0$, we
also have
    \begin{equation*}
        \mbox{div} \Gamma \mathbf{v} (\mathbf{x}_I) = \|\nabla^2
\mathbf{v}\|_{C^\alpha}O(h^{1+\alpha}),
    \end{equation*}
and similarly from mean value theorem $$ (D^* \Gamma \mathbf{v})_I
=\|\nabla^2 \mathbf{v}\|_{C^{\alpha}}O(h^{1+\alpha}).
$$
Similarly, we suppose $q\in C^{1,\alpha}$ and we have
    \begin{equation*}
        (\nabla \Gamma q)(\mathbf{x}_I) = \nabla q(\mathbf{x}_I)
        + \|\nabla q\|_{C^\alpha}O(h^\alpha),
    \end{equation*}
and from mean value theorem
    \begin{equation*}
    \|(D \Gamma q )_I- \nabla \Gamma q(\mathbf{x}_I) \| \leq C\|\nabla q\|_{C^{\alpha}} h^\alpha.
    \end{equation*}
Since $-\Delta \mathbf{v}(\mathbf{x}_I)+\nabla q(\mathbf{x}_I) =
\mathbf{f}(\mathbf{x}_I )$ and $\mathbf{f} \in C^{\alpha}$ we have
the error equation for virtual interpolation method,
\begin{align}\label{errorequation}
 -A(U-\Gamma \mathbf{v}) +D(P-\Gamma q) &=( F- \Gamma \mathbf{f}) + (\|\nabla^2 \mathbf{v}\|_{C^{\alpha}}+\|\nabla q\|_{C^\alpha}) O(h^\alpha),    \\
D^*(U-\Gamma\mathbf{v}) &=\|\nabla^2
\mathbf{v}\|_{C^\alpha}O(h^{1+\alpha})=\| \mathbf{f}
\|_{C^\alpha}O(h^{1+\alpha}).\nonumber
\end{align}
where $(U, P)$ are solution of the discrete Stokes equations \eqref{discretestokes}.

\begin{thm}\label{theorem:stability}
We suppose that all the node sets satisfy the conditions in Theorem
\ref{theorem:inf-sup_for_vip}, and the reproducing degree $m\geq 2$.
We also assume that $\mathbf{f}\in C^{\alpha}$. We let
$(\mathbf{v},q)$ the true solution in
$C^{2,\alpha}_{per}(\Omega)\times C^{1,\alpha}_{per}(\Omega)$ and $(U,P)$ discrete solution. Then, there is an absolute constant $C$
such that
    \begin{equation*}
        \| D( U - \Gamma \mathbf{v}) \|_{l^2} + \| P-\Gamma q\|_{l^2} \leq C
        h^{\alpha-1} \|\mathbf{f} \|_{C^{\alpha}(\Omega)}.
    \end{equation*}
\end{thm}

\begin{proof}
If we have $\mathbf{f}\in C^{\alpha}$, from Calderon-Zygmund theory
of Stokes equations we have $\mathbf{v}\in C^{2,\alpha}$ and $q\in
C^{1,\alpha}$. From Sobolev embedding we have
    \begin{equation*}
        \|\nabla^2 \mathbf{v}\|_{C^\alpha}+\|\nabla q\|_{C^\alpha} \leq
        C \|\mathbf{f}\|_{C^\alpha},
    \end{equation*}
and in case $\mathbf{f}\in C^{\alpha}$, we have
    \begin{equation*}
        \|F- \Gamma \mathbf{f} \|_{l^\infty} \leq Ch^{\alpha}
        \|\mathbf{f}\|_{C^{1,\alpha}}.
    \end{equation*}
The $(U,P)$ satisfies the discrete Stokes equations ($\ref{discretestokes}$) and therefore we have error equation \eqref{errorequation} for $(E,R) = (U- \Gamma \mathbf{v}, P-\Gamma q)$ such that at each $\mathbf{x}_I$
    \begin{align*}
        (-AE+D R)_I = O(h^\alpha)
        \|\mathbf{f}\|_{C^{\alpha}}\quad\mbox{and}\quad
        D^* E_I = O(h^{1+\alpha})
    \|\mathbf{f}\|_{C^{\alpha}}.
    \end{align*}
By the discrete Poincar\'e inequality (see~\cite{le2013}), with the
condition $\sum_I E_I =0$, we have
    \begin{equation*}
        \|E\|_{l^2} \leq C\| D E\|_{l^2}.
    \end{equation*}
Then simply applying $E$ to error equation and considering
ellipticity of discrete operator $A$, we prove that
    \begin{equation}\label{error}
        \|D E\|_{l^2}^2 \leq C h^{\alpha -1} \|\mathbf{f}\|_{C^{\alpha}} \|R\|_{l^2}
        + Ch^{2\alpha-2} \|\mathbf{f}\|^2_{C^{\alpha}}.
    \end{equation}
From inf-sup condition (see \ref{theorem:inf-sup_for_vip}), we find $W\in \mathbb{R}^{2M}$ such that
    \begin{equation*}
        \|R\|_{l^2}\leq C\frac{\langle D^* W,R \rangle}{\|D W\|_{l^2}}.
    \end{equation*}
Applying $W$ to error equation and we obtain from inf-sup condition
    \begin{equation*}
        C \|R\|_{l^2} \|D W\|_{l^2}\leq \langle D^* W,R \rangle
    =\langle DW,DE \rangle+O(h^{\alpha-1})
    \|\mathbf{f}\|_{C^{\alpha}} \|D W\|_{l^2}.
    \end{equation*}
Therefore we conclude
    \begin{align*}
         \|R\|_{l^2} &= \frac{1}{\|D W\|_{l^2}} \left(\langle DW,DE \rangle+
         O(h^{\alpha-1})\|\mathbf{f}\|_{C^{\alpha}}
         \|D W\|_{l^2}\right) \\
         &\leq C \|DE\|_{l^2} +C h^{\alpha-1} \|\mathbf{f}\|_{C^{\alpha}.
}
    \end{align*}
and from Cauchy-Schwarz inequality on the error of pressure $R$ term of
(\ref{error}) we prove the theorem.
\end{proof}

\begin{remark}
    Our theorem, \ref{theorem:stability} implies that
    \begin{align*}
    h^2\| D( U - \Gamma \mathbf{v}) \|^2_{l^2} &+ h^2\| P-\Gamma q\|^2_{l^2}
    \\
    &=
    \int_\Omega
\left|\overline{D( U - \Gamma \mathbf{v})}\right|^2 + \left|\overline{P - \Gamma
\mathbf{q} }\right|^2 d\mathbf{x}   \\
    &\leq
    Ch^{2\alpha} \|\mathbf{f}\|_{C^{\alpha}}^2.
    \end{align*}
\end{remark}

\section{Numerical examples}

In this section we present a series of test problems of increasing
complexity to demonstrate the accuracy and robustness of the VIP
method.

\subsection{Spatial convergence test}

We consider the Kovasznay flow, which is steady problem with
analytic expression. The velocity and pressure fields are given by
the following equations,
\begin{align*}
    u(x, y) &= 1 - e^{\lambda x}\cos(2\pi y), \\
    v(x, y) &= \frac{\lambda}{2\pi}e^{\lambda x}\sin(2\pi y), \\
    p(x, y) &= \frac{1}{2}\left(1 - e^{2\lambda x}\right),
\end{align*}
where $\lambda = \frac{Re}{2} -\left(\frac{Re^2}{4} +
4\pi^2\right)^{1/2} $ with $Re = 40$. We consider the Kovasznay flow
on the domain $\Omega = [-0.5, 1.5] \times [0, 2]$, which is
discretized with regular nodes. Fig.~\ref{covg:Kovasznay}(d) shows
the discrete norms of the errors in the velocity and pressure with
the analytical solutions. The contour lines for $u$-velocity,
$v$-velocity, and pressure are shown in
Fig.~\ref{covg:Kovasznay}(a)-(c).

\begin{figure}[H]
    \centering
    \subfigure[]{\includegraphics[width=0.45\textwidth]{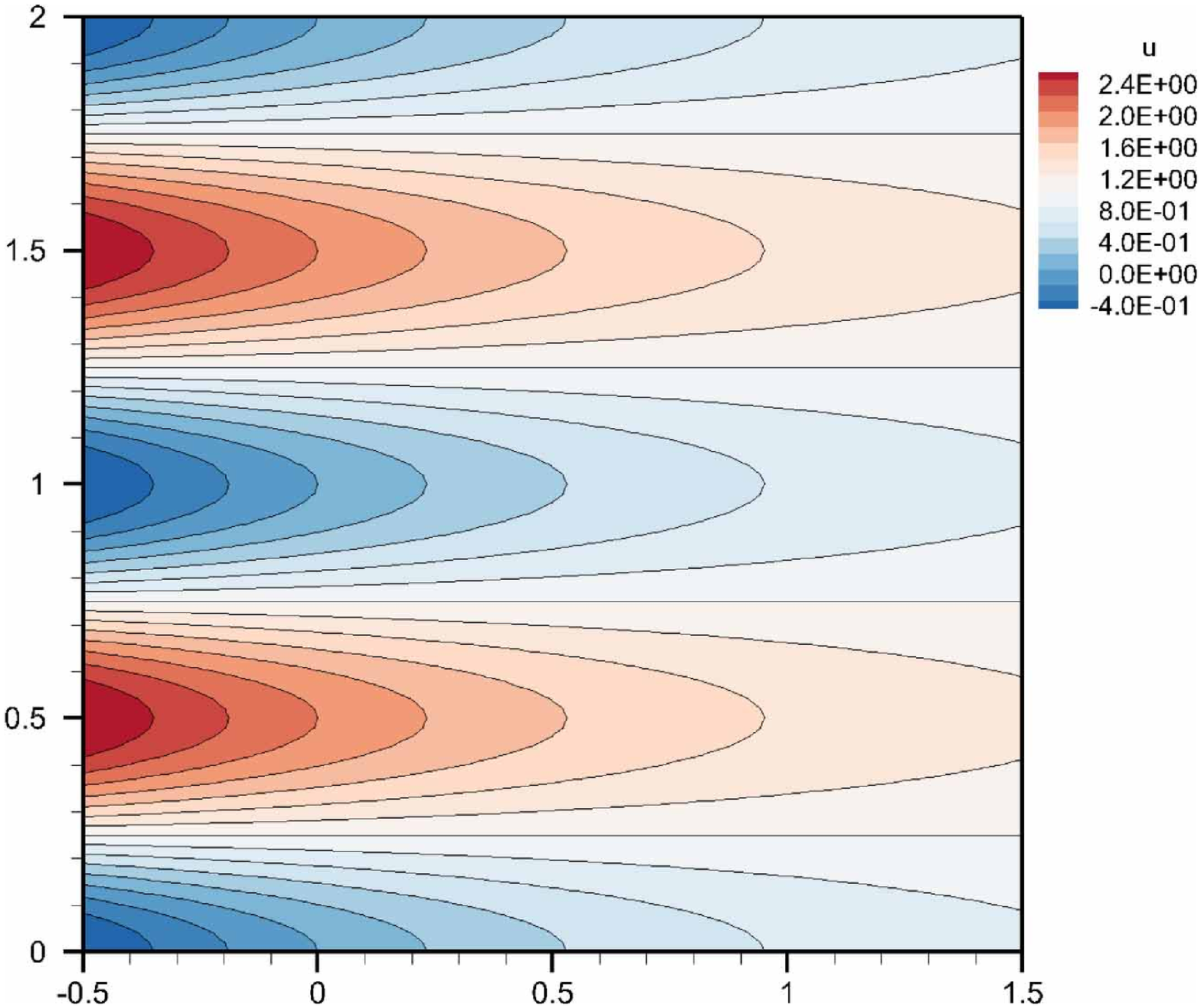}}
    \subfigure[]{\includegraphics[width=0.45\textwidth]{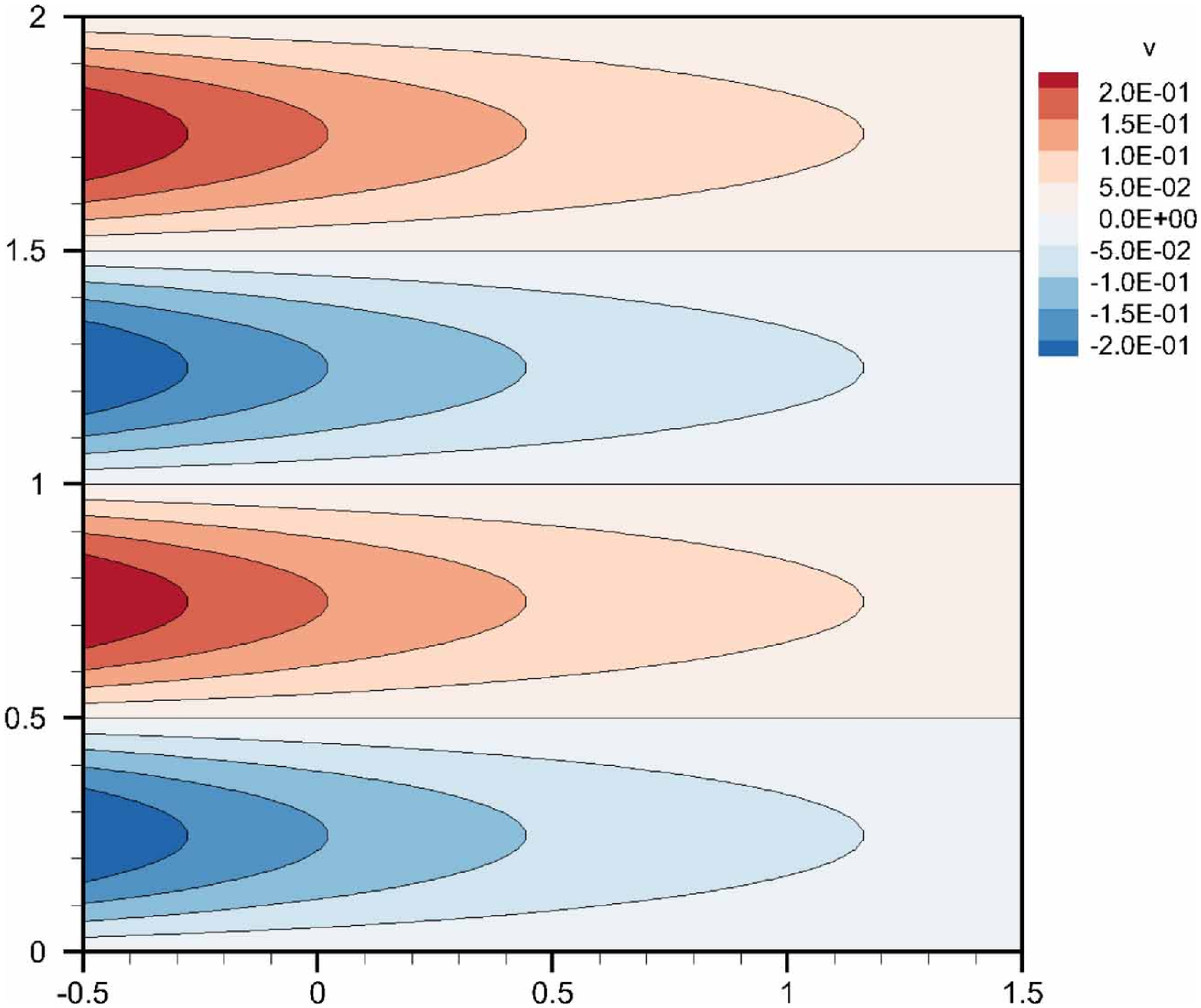}}
    \subfigure[]{\includegraphics[width=0.45\textwidth]{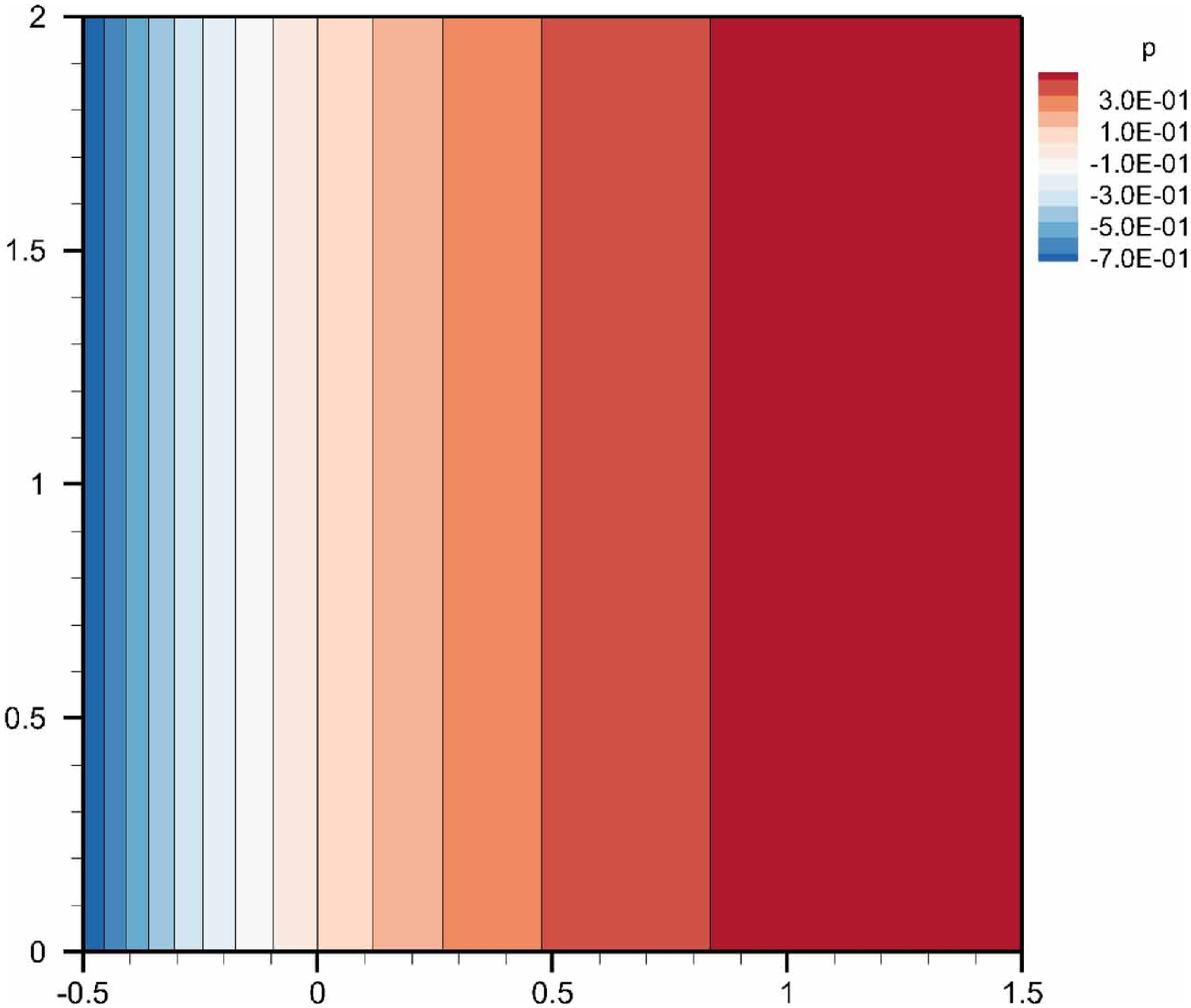}}
    \subfigure[]{\includegraphics[width=0.45\textwidth]{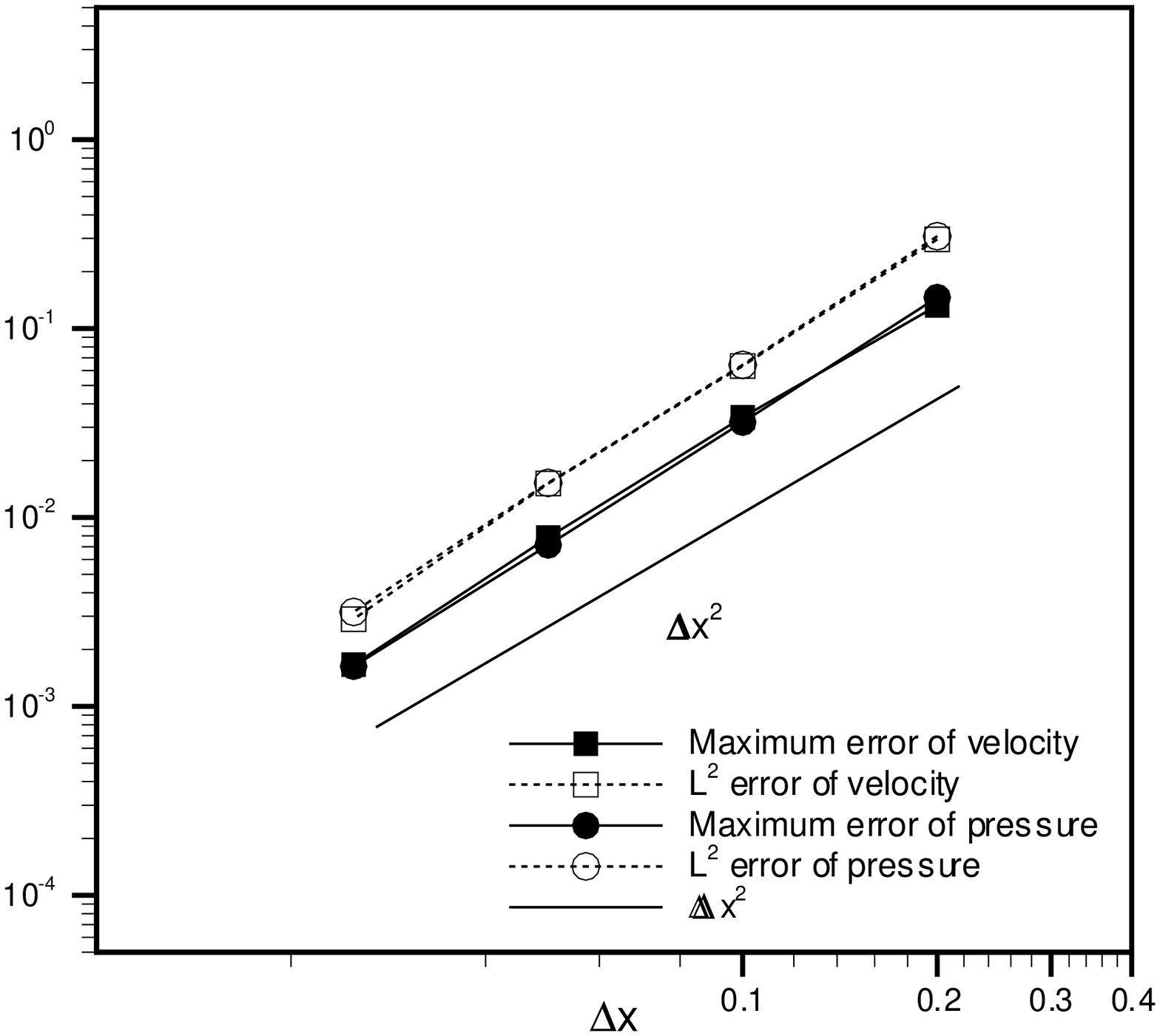}}
    \caption{Kovasznay flow : (a) $u$-velocity; (b) $v$-velocity; (c) pressure; (d) the convergence of the numerical solutions from the uniform nodes.}\label{covg:Kovasznay}
\end{figure}

\subsection{Lid-driven cavity flow}
The next test is a two-dimensional lid-driven cavity problem on the
domain $\Omega = [0, 1]\times [0,1]$ with $(u,v) = (1,0)$ on the top
and no-slip boundary conditions on the rest part of the boundary.
Figure~\ref{cavity:re100}(d) and Figure~\ref{cavity:re400}(d) show
the centerline velocities $u(y)\mbox{ and }v(x)$ along the vertical
and horizontal centerlines, respectively. Reynolds numbers of $Re =
100 \mbox{ and } 400$ are chosen for validating the current method.
The present result is in good agreement with that of Ghia \textit{et
al.}~\cite{ghia} who used 128$\times$128 uniformly distributed
rectangular cells. The contour lines for stream function, pressure,
and vorticity are shown in Fig~\ref{cavity:re100}(a)-(c) and
Fig~\ref{cavity:re400}(a)-(c).

\begin{figure}[H]
    \centering
    \subfigure[]{\includegraphics[width=0.45\textwidth]{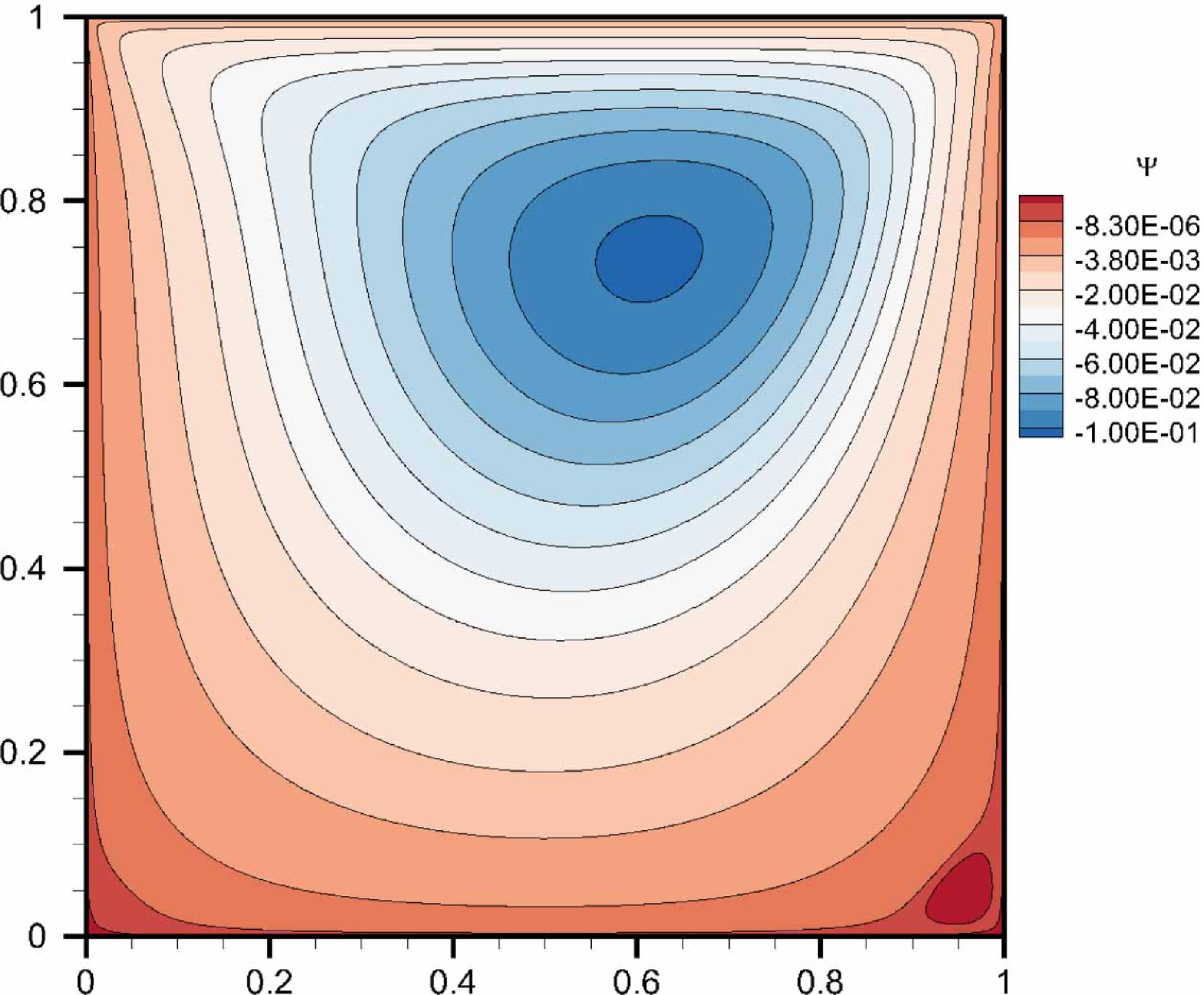}}
    \subfigure[]{\includegraphics[width=0.45\textwidth]{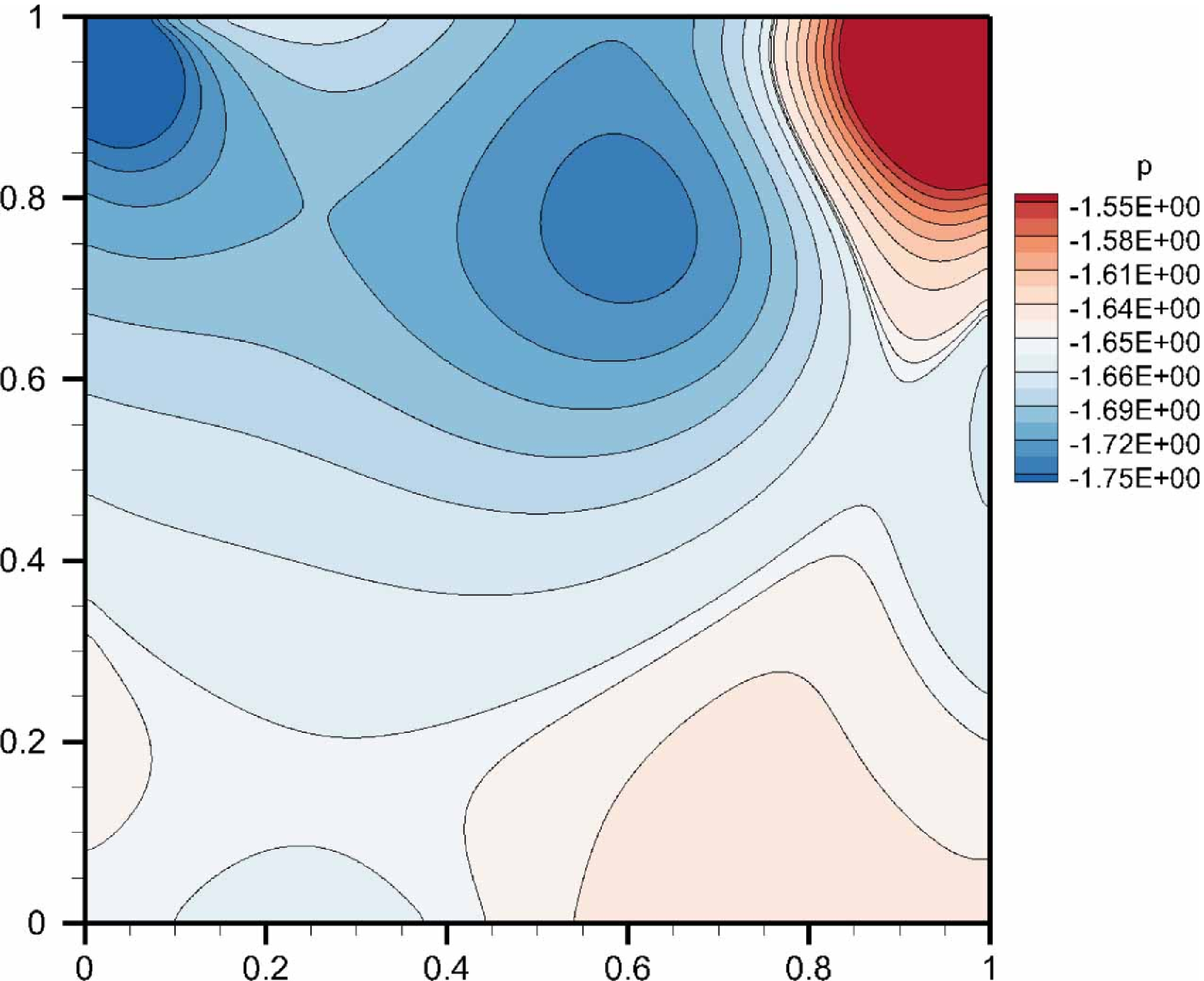}}
    \subfigure[]{\includegraphics[width=0.45\textwidth]{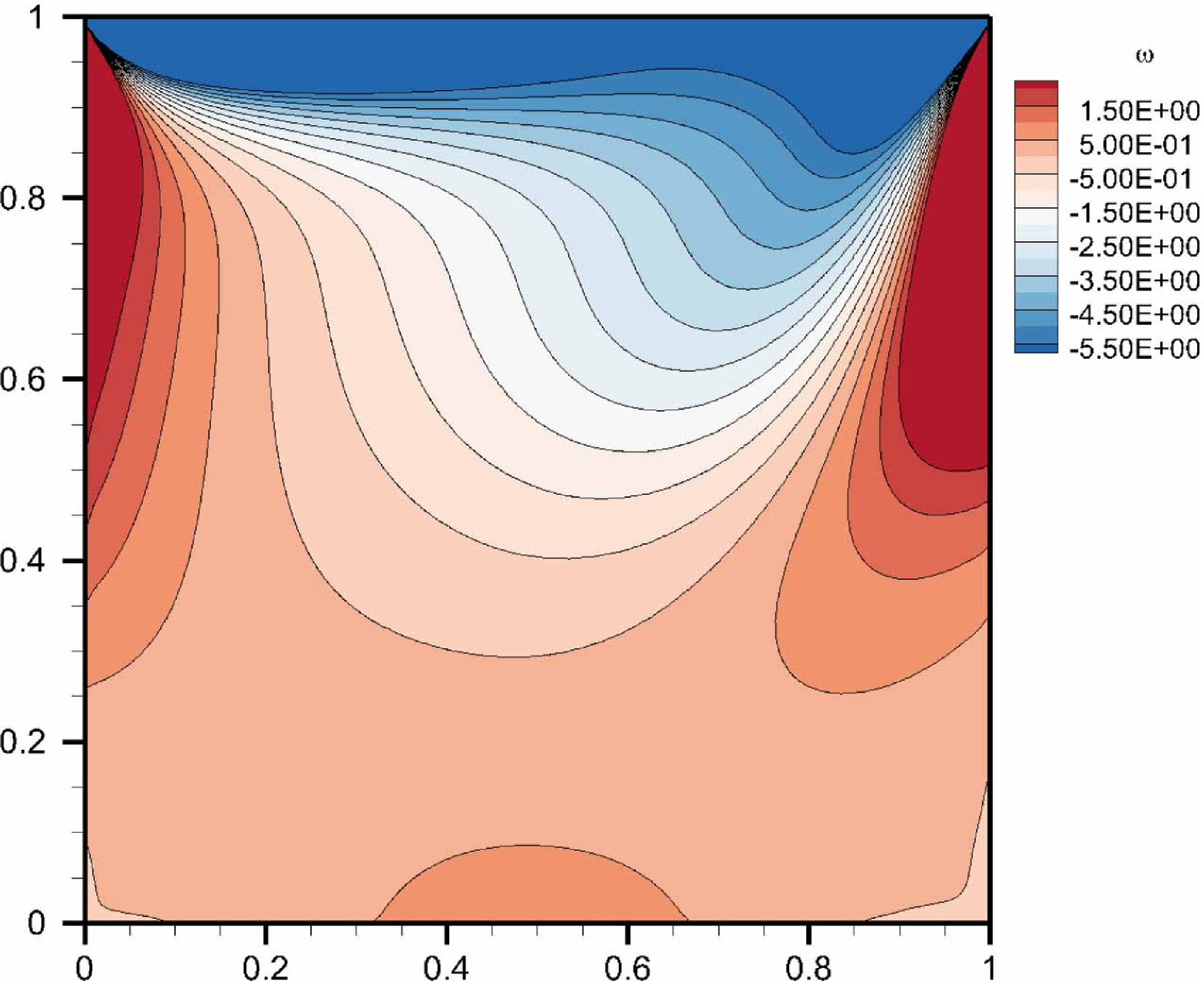}}
    \subfigure[]{\includegraphics[width=0.45\textwidth]{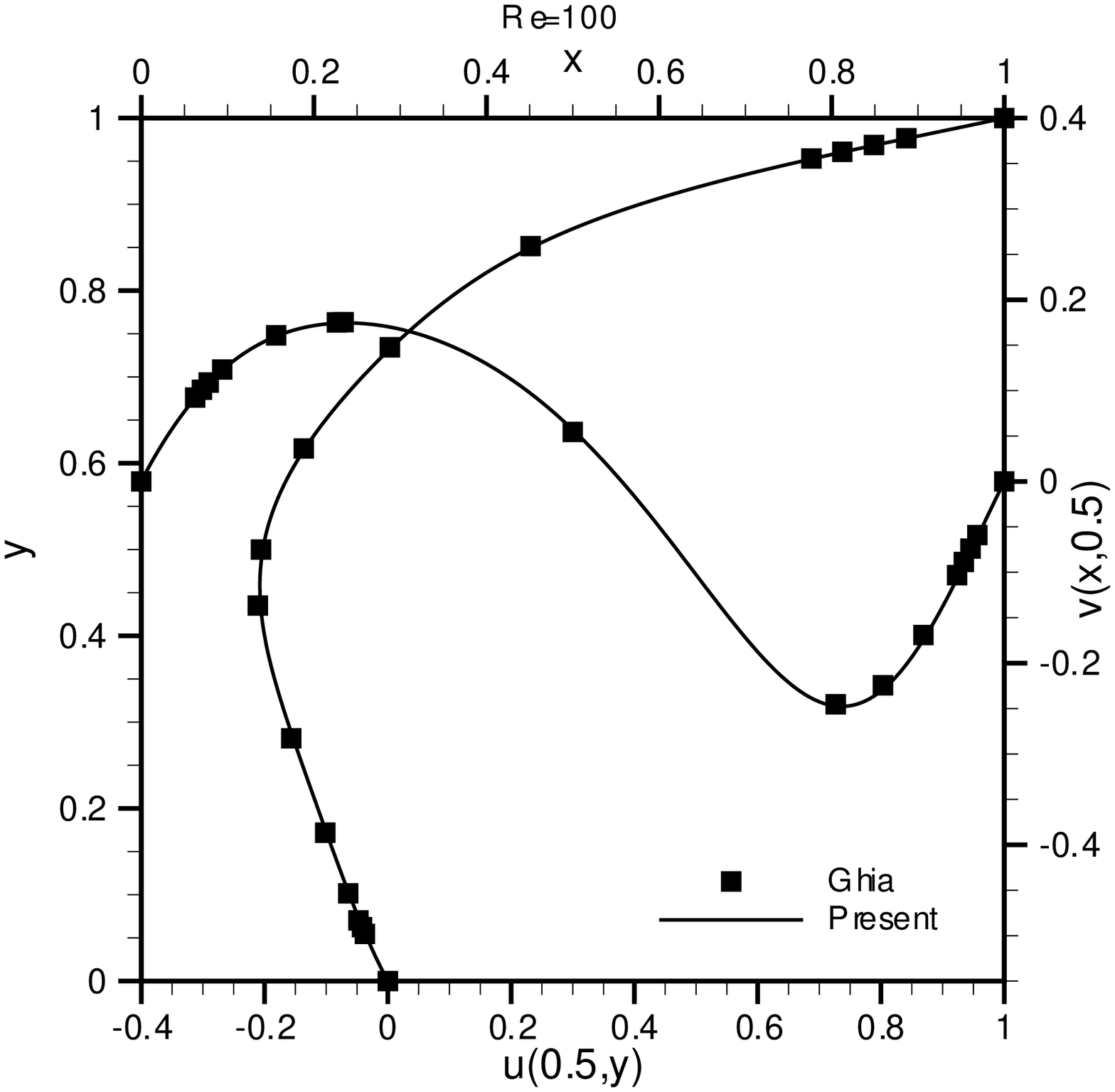}}
    \caption{Lid-driven cavity flow with $Re=100$ : (a) stream function; (b) pressure; (c) vorticity; (d) centerline velocities $u$ and $v$. Results from Ghia \textit{et
al.}~\cite{ghia} are compared with current numerical solutions.}
\label{cavity:re100}
\end{figure}

\begin{figure}[H]
    \centering
    \subfigure[]{\includegraphics[width=0.45\textwidth]{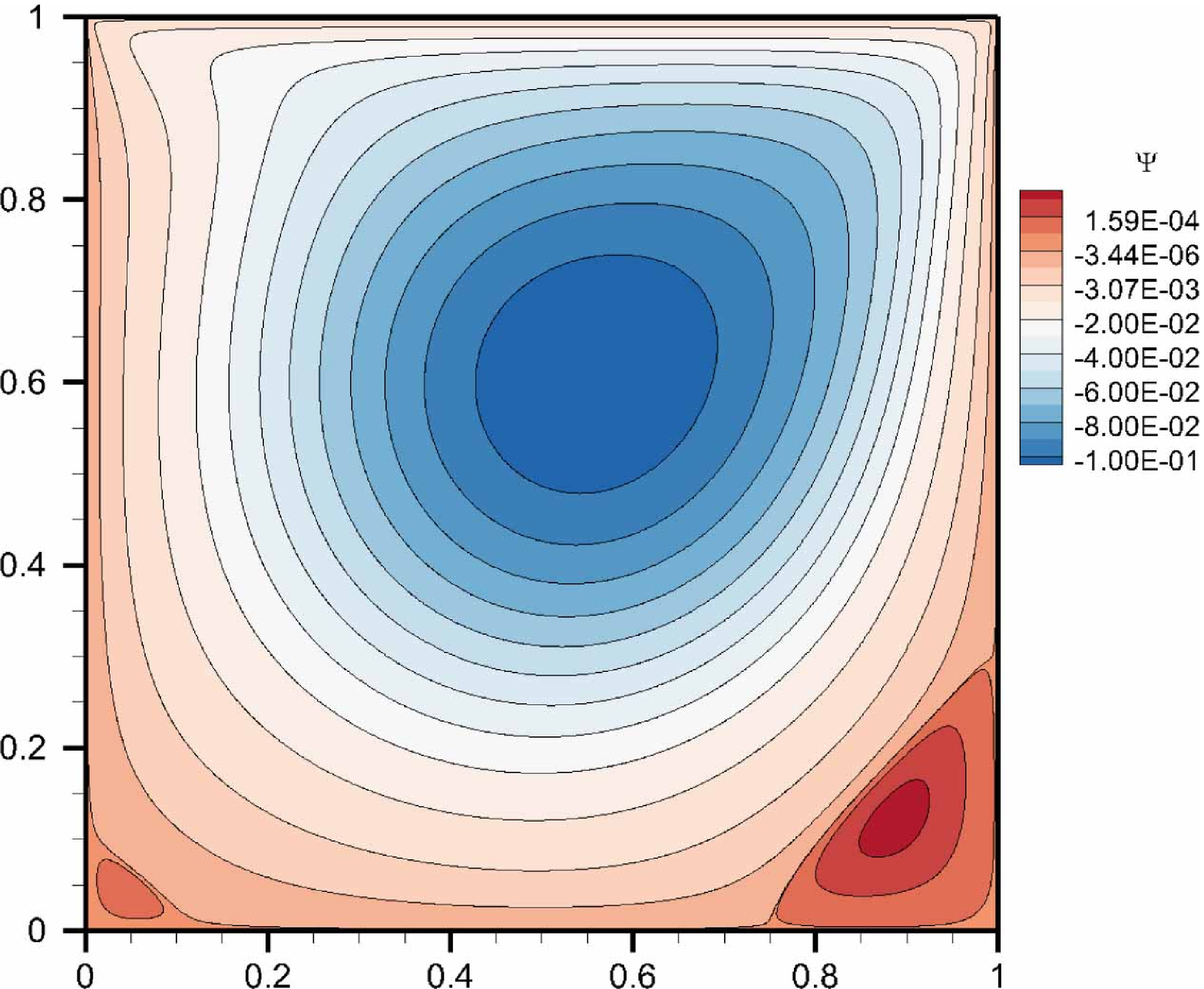}}
    \subfigure[]{\includegraphics[width=0.45\textwidth]{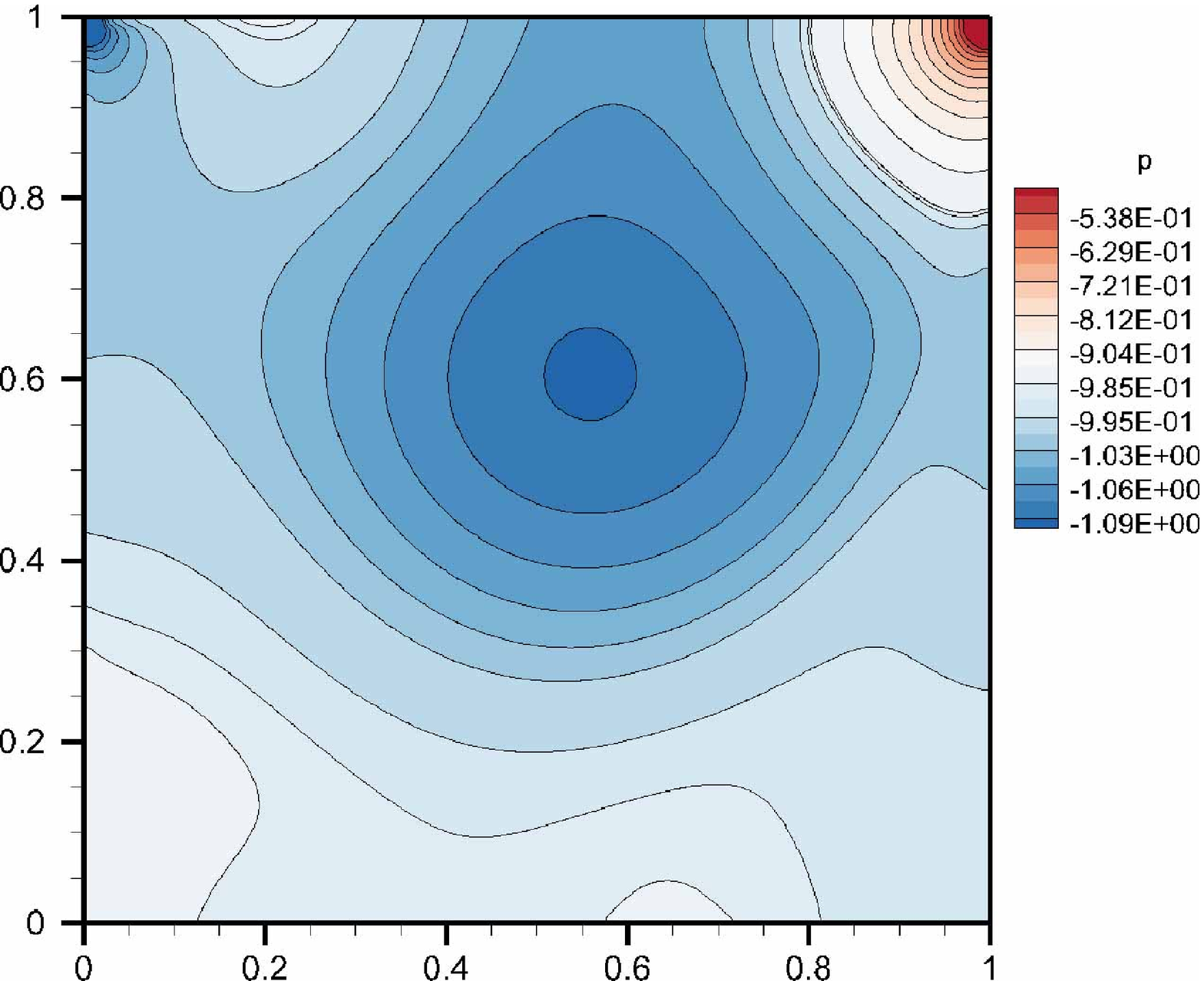}}
    \subfigure[]{\includegraphics[width=0.45\textwidth]{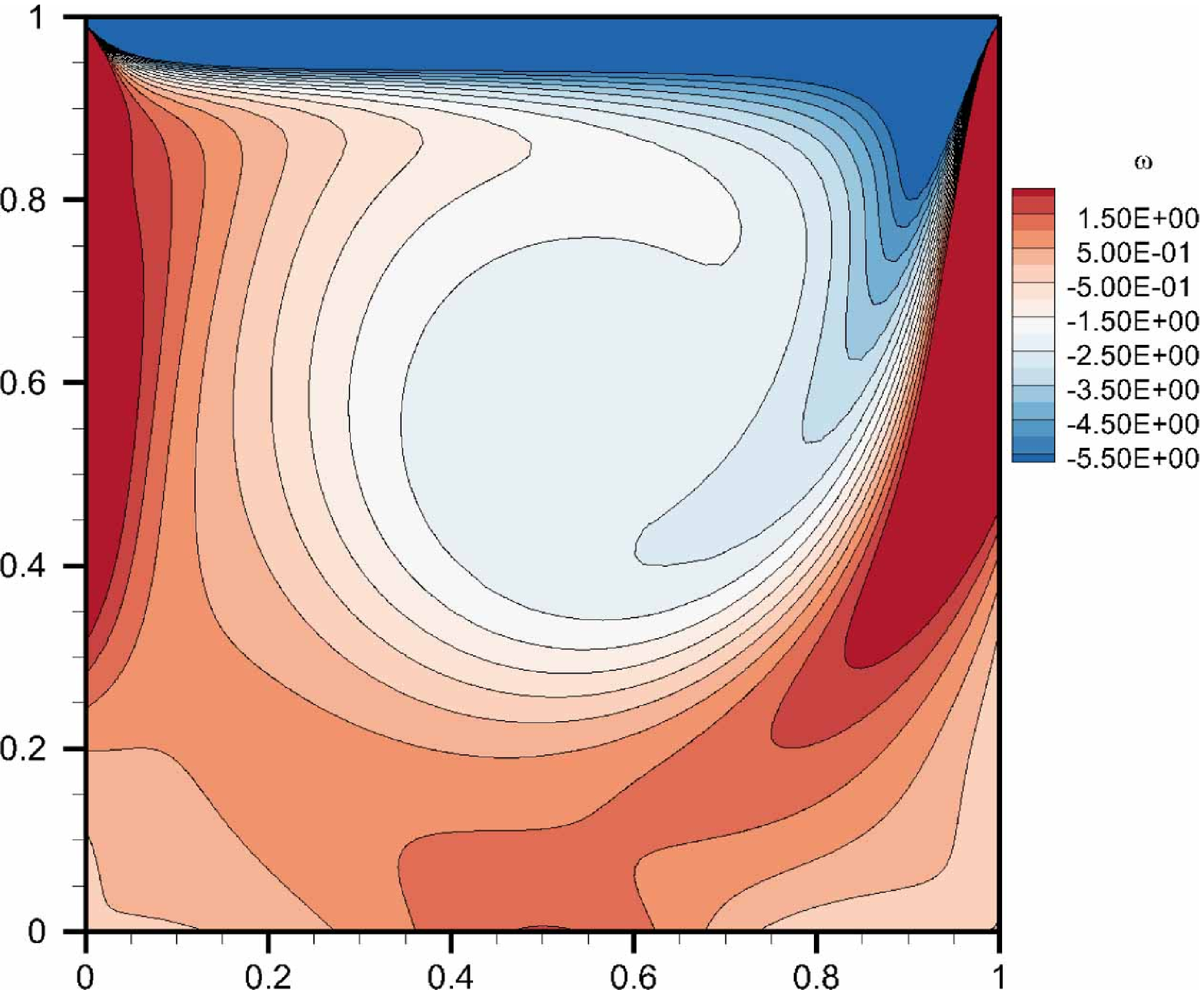}}
    \subfigure[]{\includegraphics[width=0.45\textwidth]{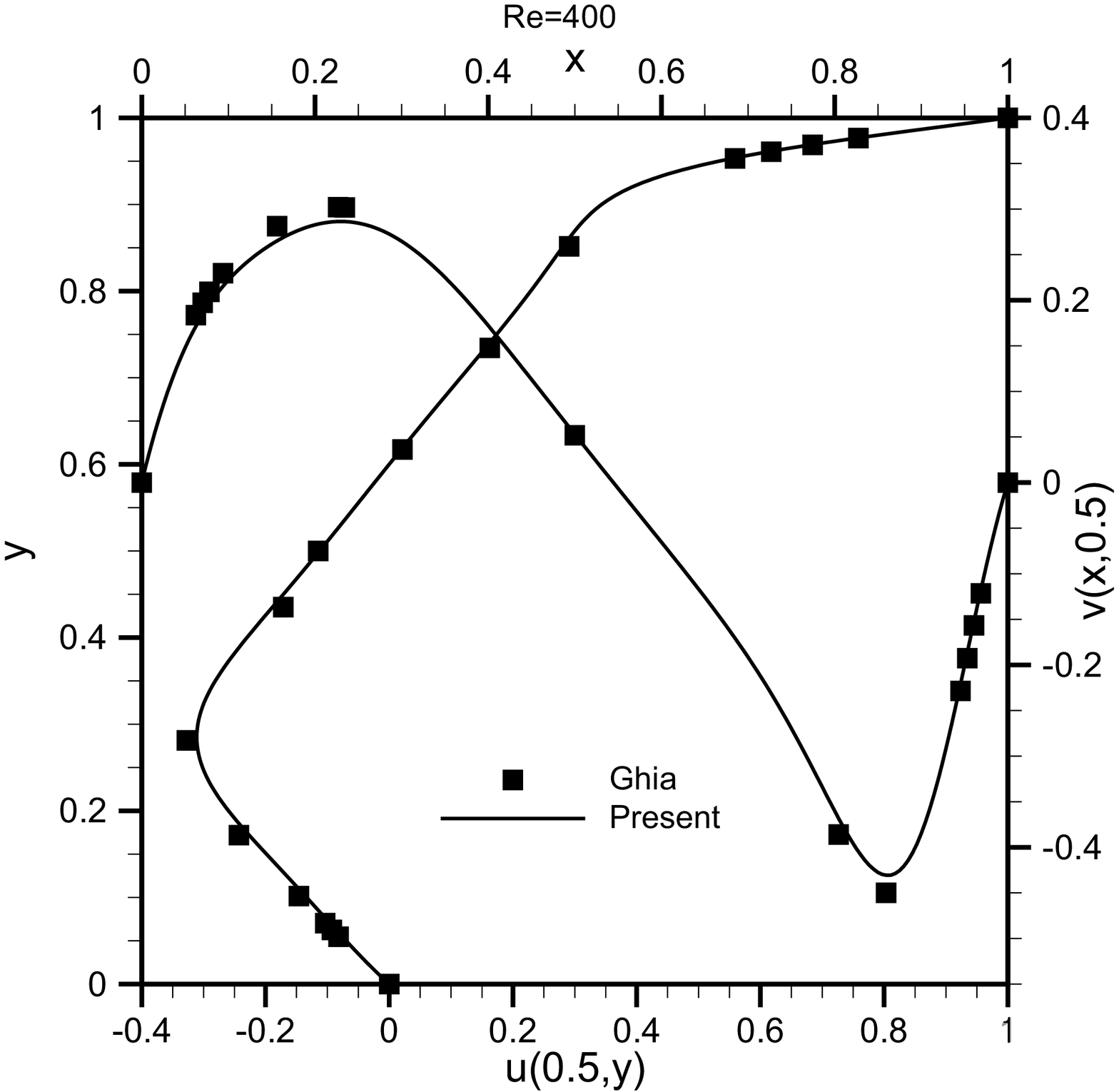}}
    \caption{Lid-driven cavity flow with $Re=400$ : (a) stream function; (b) pressure; (c) vorticity; (d) centerline velocities $u$ and $v$. Results from Ghia \textit{et
al.}~\cite{ghia} are compared with current numerical
solutions.}\label{cavity:re400}
\end{figure}

\subsection{Flow over a circular cylinder}
We consider flow over a circular cylinder as another test problem
because the dimensions of the recirculation zone and the force on
the cylinder at various Reynolds numbers are readily available from
previous experimental and numerical studies. Our two-dimensional
simulations are performed by introducing a cylinder of diameter d =
1 in a large computational domain D with initially uniform flow, $u
= u_{\infty} = 1$. Reynolds numbers of $Re = 10,\,20, \mbox{ and }
40$ are chosen for validating the current method at steady-state.
The resulting wake dimensions and drag coefficients are compared to
values reported in the
literatures~\cite{dennis, takami,fornberg,ding,kim}. In
Fig.~\ref{fig::circular_cp}, the vorticity and the pressure
coefficient $C_p$ on the body surface are plotted, while
Table~\ref{CYL_CD:Re102040} shows the drag coefficient($C_D$) for
each Reynolds number of 10, 20, and 40. The stream function and
vorticity contours around the body are also illustrated in Fig.
\ref{fig:circular_stream_vorticity}.

    \begin{table}
        \centering
        \begin{tabular}{lrrr}
            \hline
            $C_D$\hspace{2.5cm}         & $Re=10$                       & \hspace{0.5cm} $Re=20$    & \hspace{0.5cm}$Re=40$  \\
            \hline\hline
            Dennis et al.\cite{dennis}  & 2.85                              & 2.05                                  & 1.522 \\
            Takami et al.\cite{takami}  & 2.80                              & 2.01                                  & 1.536 \\
            Tuann et al.\cite{tuann}        & 3.18                              & 2.25                                  & 1.675 \\
            Fornberg\cite{fornberg}         &                                   & 2.00                                  & 1.498 \\
            H. Ding\cite{ding}              & 3.07                              & 2.18                                  & 1.713 \\
            Kim et al.\cite{kim}                &                                   &                                           & 1.51\phantom{0}  \\
            \\
            Present                             & 3.03                          & 2.17                                  & 1.536 \\
            \hline
        \end{tabular}
        \caption{Comparison of drag coefficient for steady
        flow.}\label{CYL_CD:Re102040}
    \end{table}

 \begin{figure}[H]
        \centering
        \subfigure[Wall pressure coefficient ($C_p$).]
        {
            \includegraphics[width=0.45\textwidth]{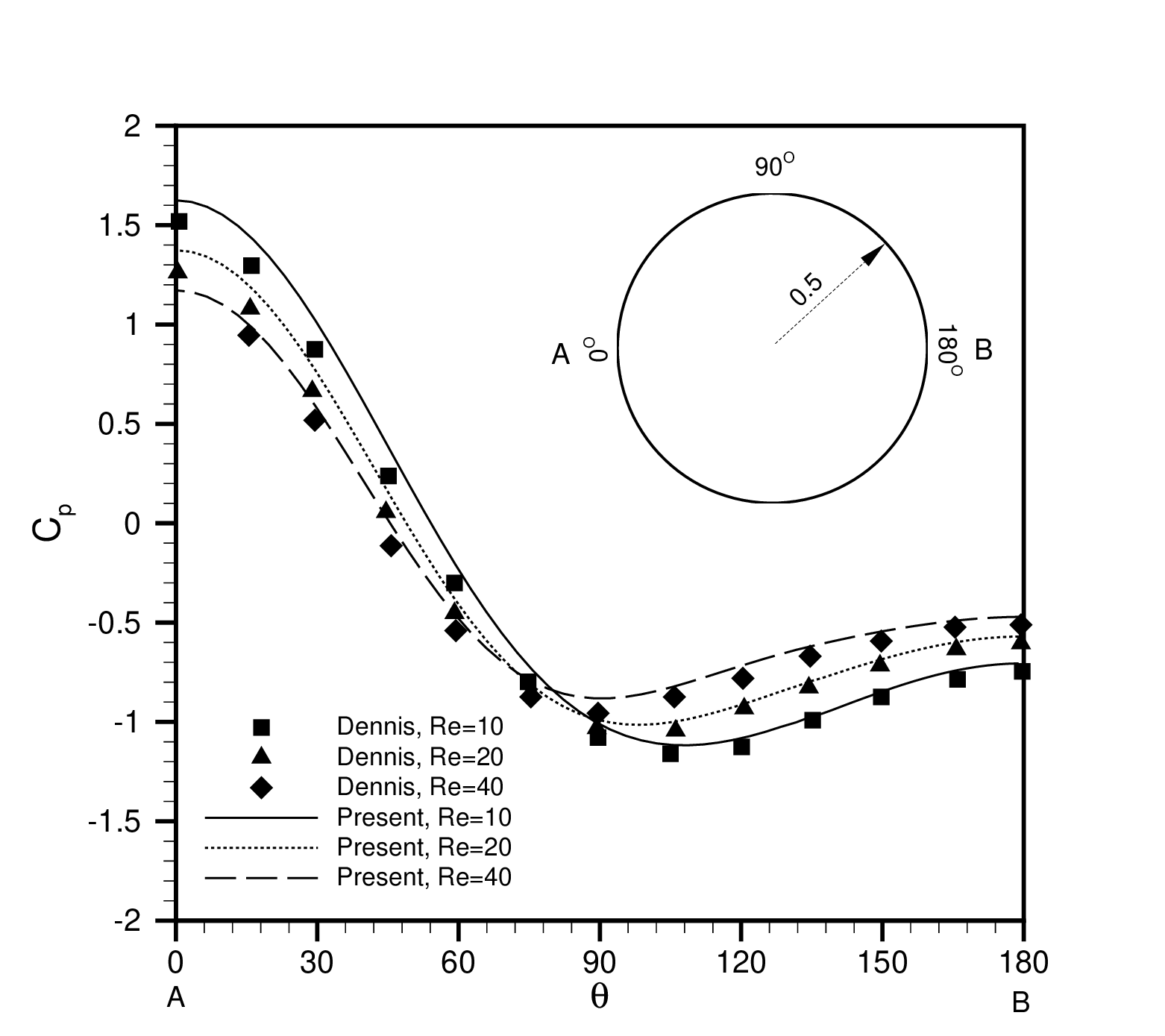}
        }
        \subfigure[Wall vorticity ($\omega$) for]
        {
            \includegraphics[width=0.45\textwidth]{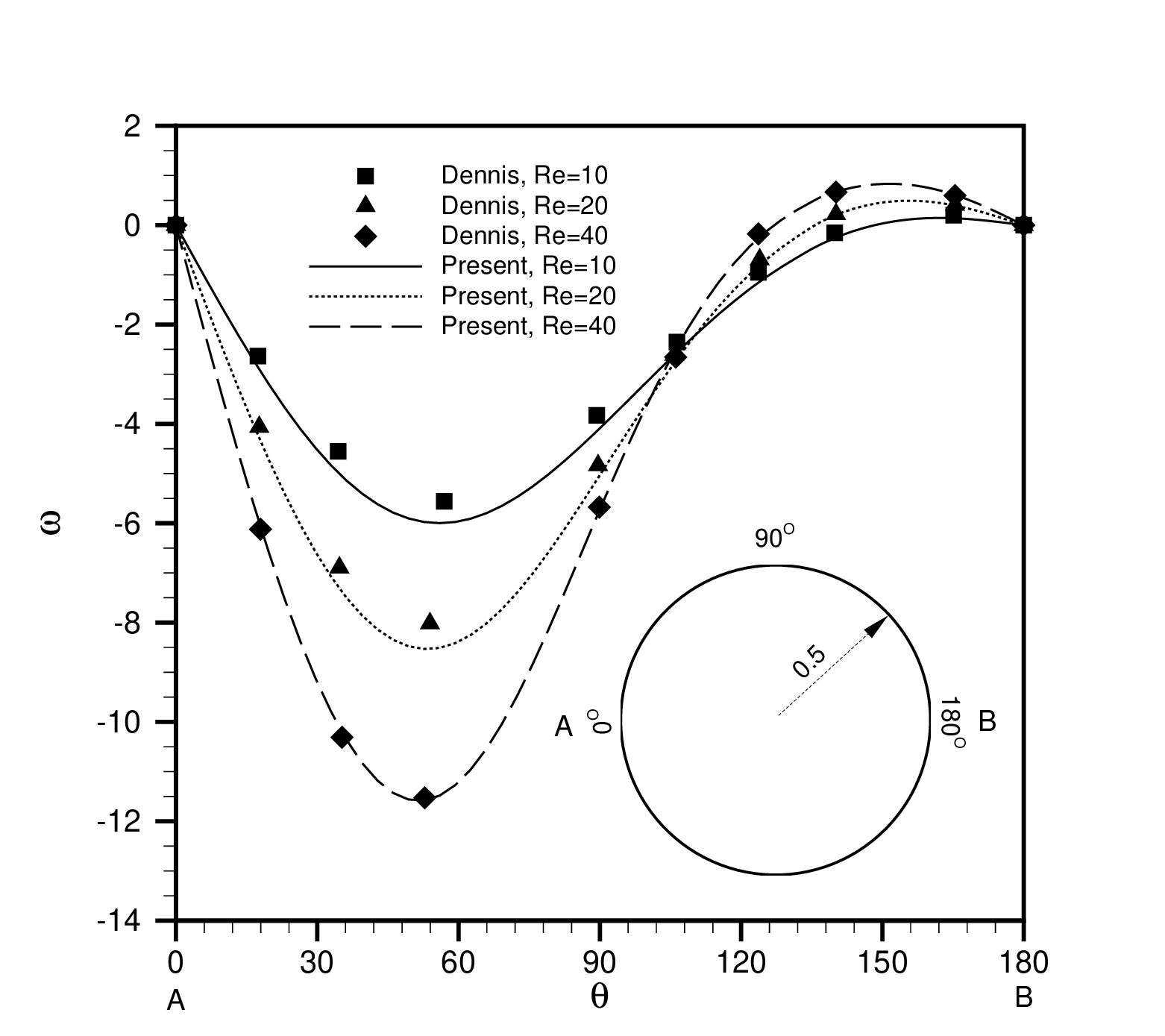}
        }
        \caption{Comparison of the vorticity and the pressure coefficients on the circular cylinder with $Re=10,20$ and $40$.}
        \label{fig::circular_cp}
    \end{figure}

   \begin{figure}[H]
        \centering
        \subfigure[Stream functions for $Re=10$]
        {
            \includegraphics[width=0.45\textwidth]{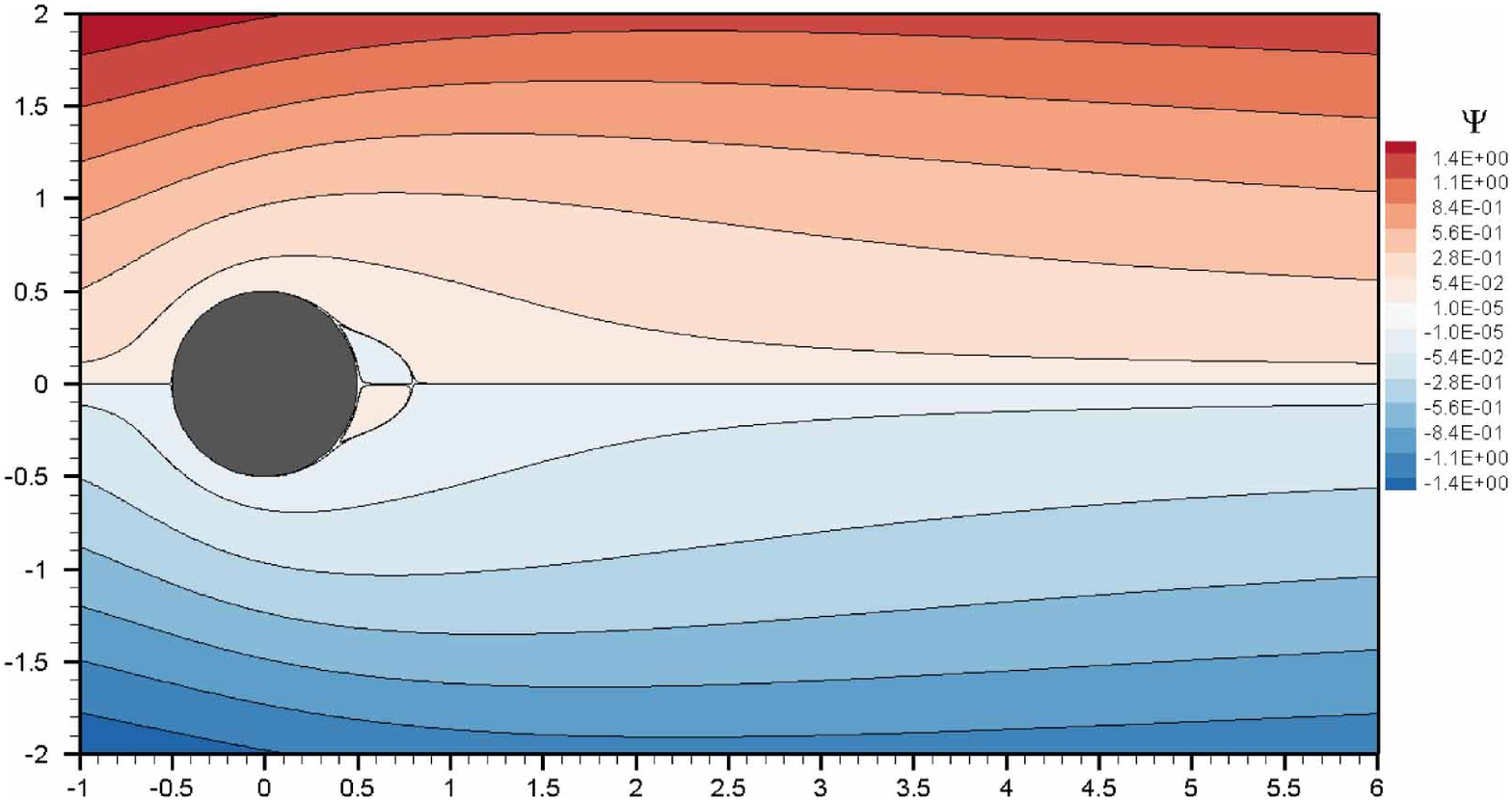}
        }
        \subfigure[Vorticity for $Re=10$]
        {
            \includegraphics[width=0.45\textwidth]{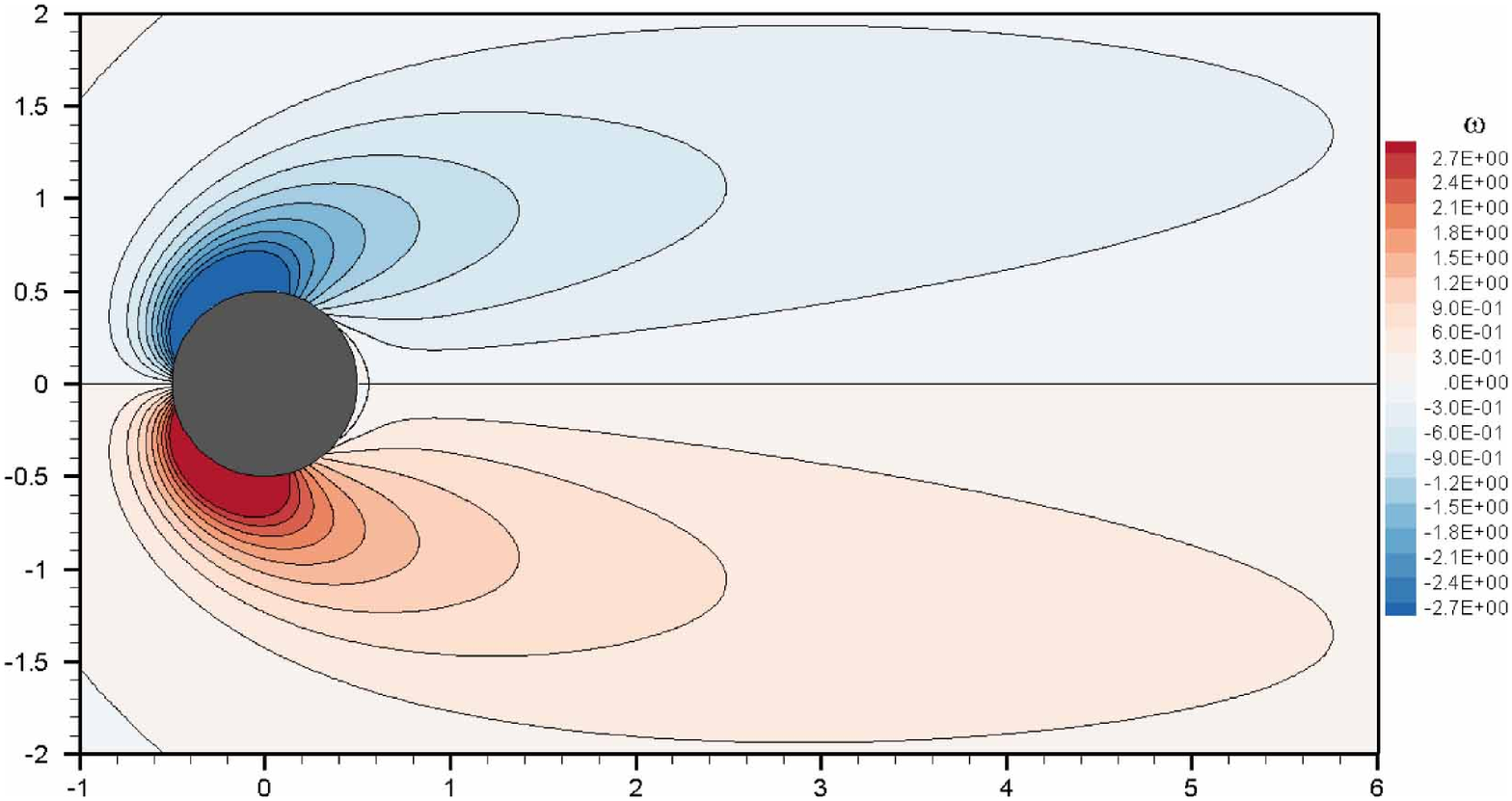}
        }
        \subfigure[Stream function for $Re=20$]
        {
            \includegraphics[width=0.45\textwidth]{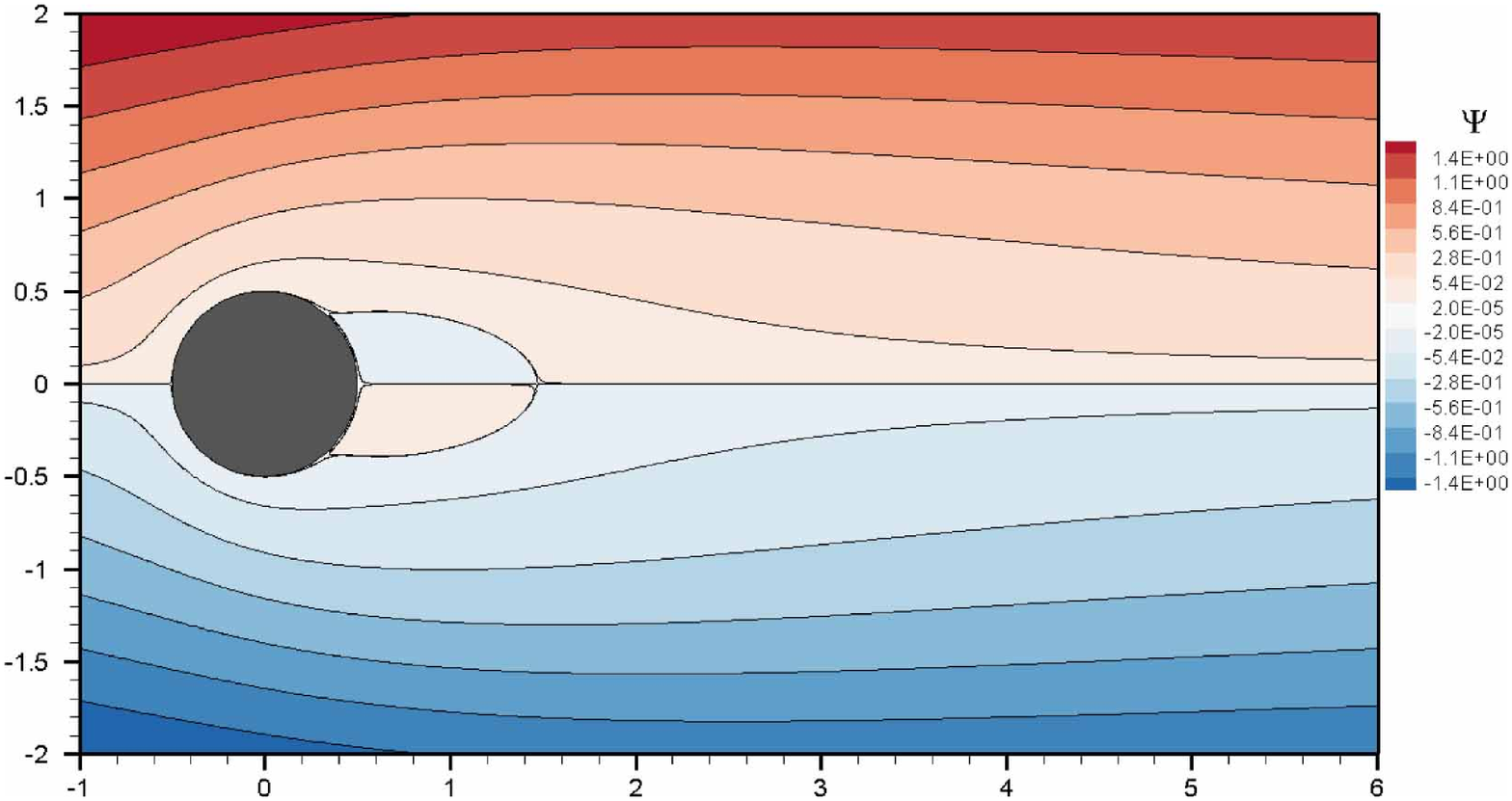}
        }
        \subfigure[Vorticity for $Re=20$]
        {
            \includegraphics[width=0.45\textwidth]{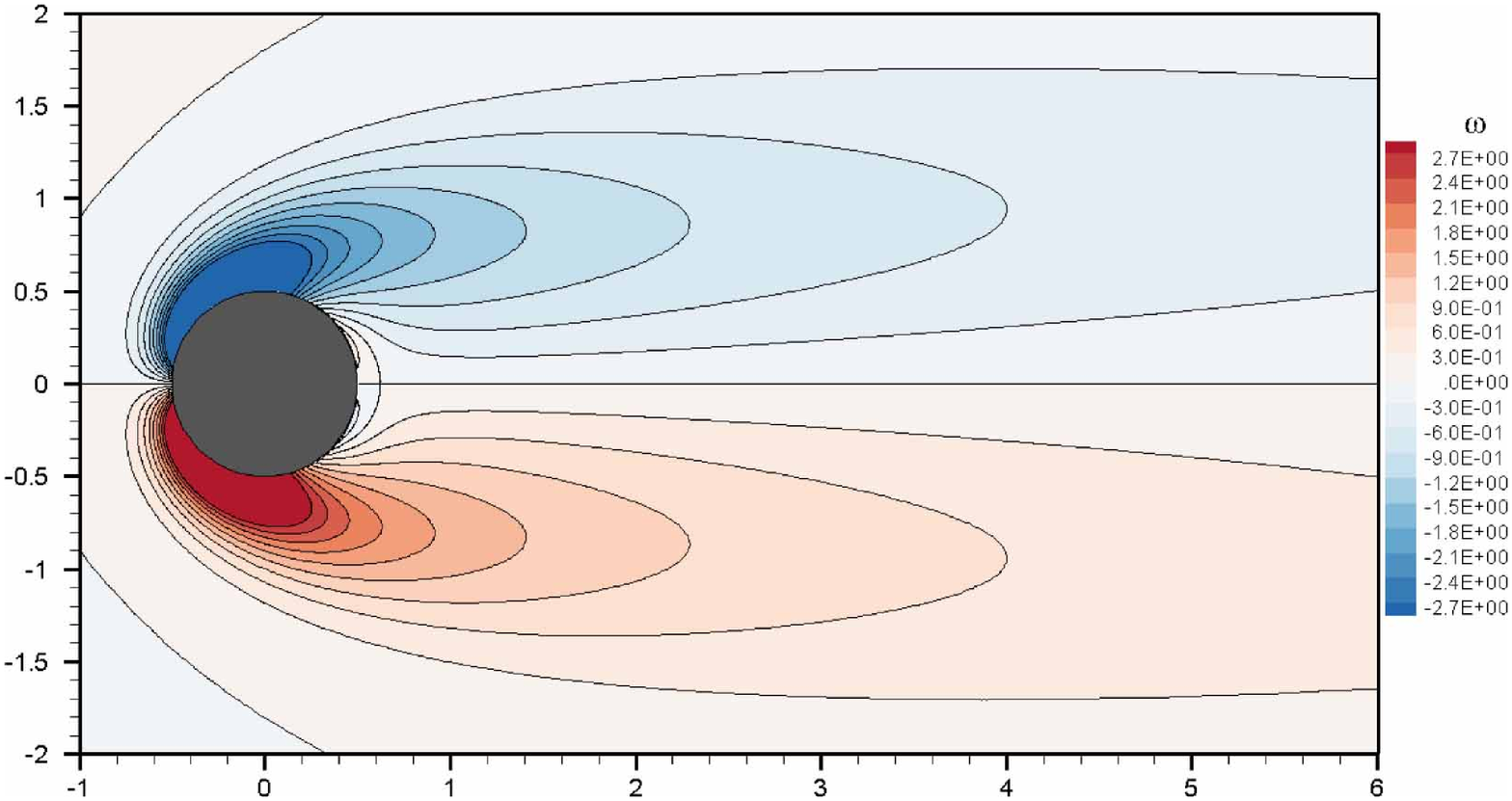}
        }
       \subfigure[Stream functions for $Re=40$]
        {
            \includegraphics[width=0.45\textwidth]{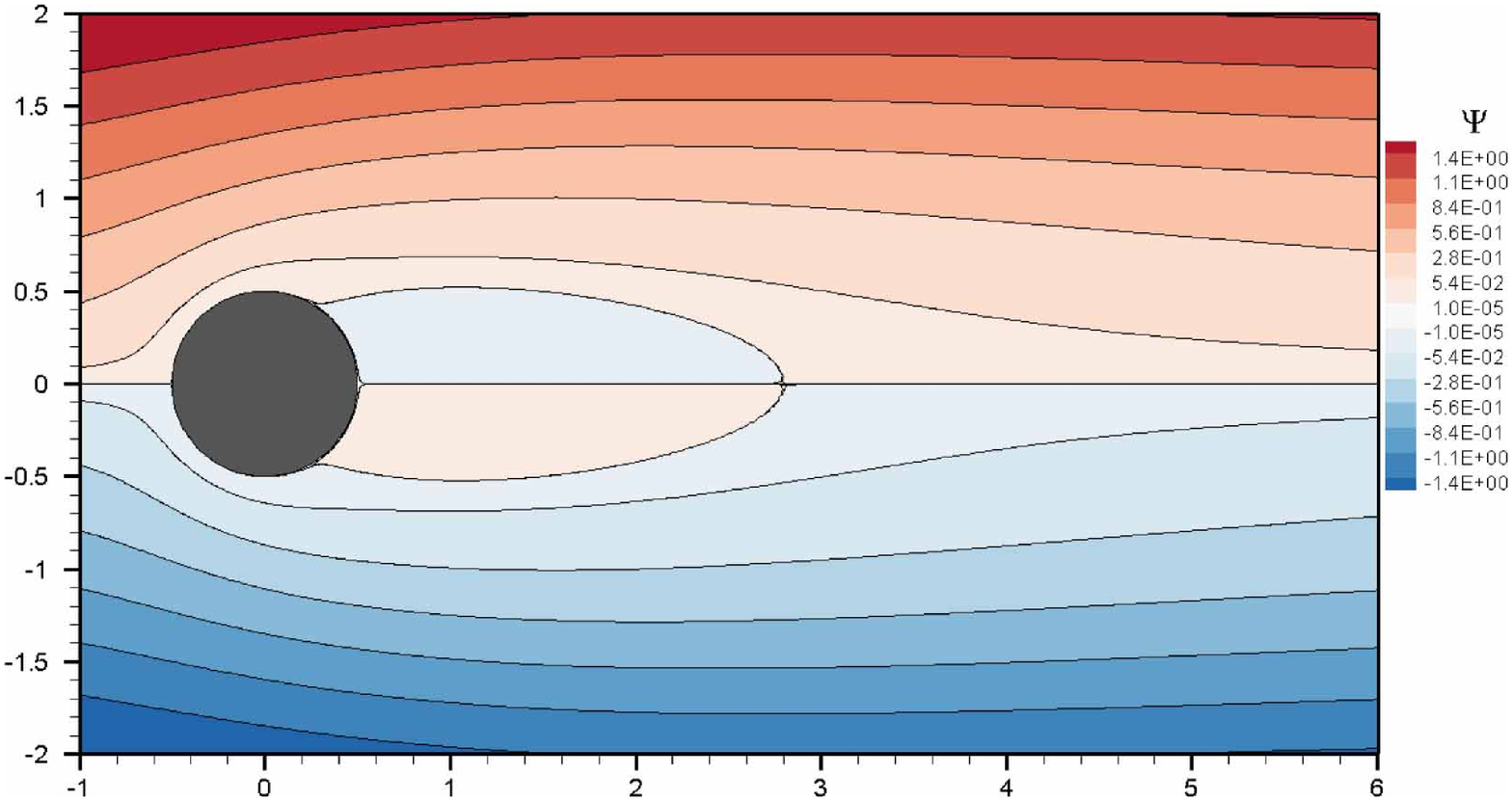}
        }
        \subfigure[Vorticity for $Re=40$]
        {
            \includegraphics[width=0.45\textwidth]{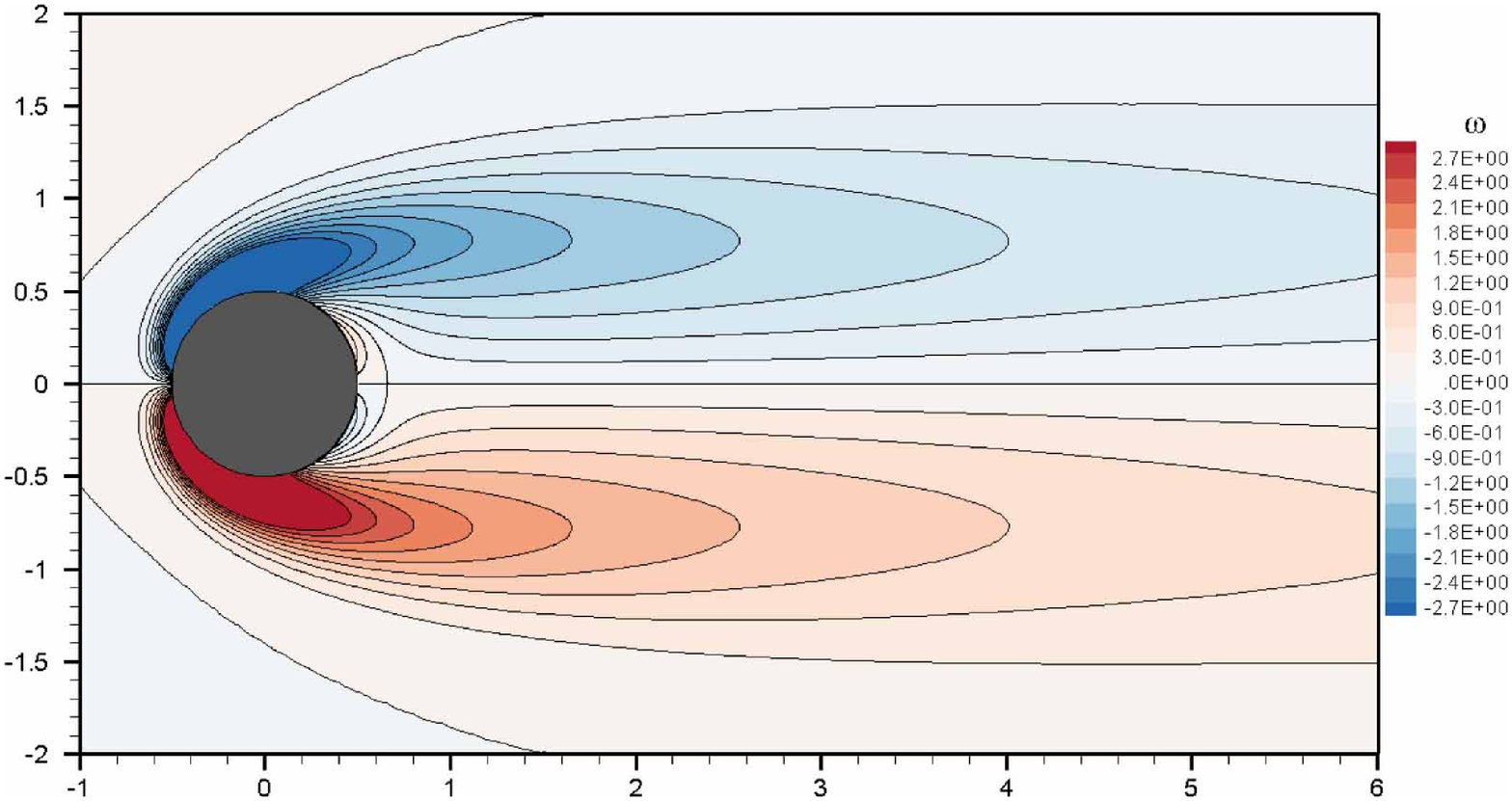}
        }
        \caption{Stream function and vorticity of flow over a circular cylinder with $Re=10, 20$, and $40$.}
        \label{fig:circular_stream_vorticity}
   \end{figure}

\bibliographystyle{siam}
\bibliography{ref}

\end{document}